\documentclass[11pt]{article}

\usepackage{amsmath,amssymb,amsfonts,textcomp,amsthm,xifthen,psfrag,graphicx,color,MnSymbol}
\usepackage{subcaption,enumitem}
\usepackage[T1]{fontenc}
\usepackage{lscape}

\usepackage{hyperref}

\date{}

\newtheorem{theorem}{Theorem}
\newtheorem{lemma}[theorem]{Lemma}
\newtheorem{cor}[theorem]{Corollary}
\newtheorem{prop}[theorem]{Proposition}
\newtheorem{remark}[theorem]{Remark}

\theoremstyle{definition} 

\def\ignore#1{}
\def\comment#1{{\tt (#1)}}

\newcommand{\R}{\mathbb{R}}
\newcommand{\Rt}{\mathbb{R}^2}
\newcommand{\Rtt}{\mathbb{R}^{2\times 2}}

\renewcommand{\SS}{\mathbb{S}}
\newcommand{\T}{\mathsf{T}}

\newcommand{\cE}{\mathcal{E}}

\newcommand{\cV}{\mathcal{V}}

\newcommand{\cS}{\mathcal{S}}

\newcommand{\bn}{n}
\newcommand{\bt}{t}

\newcommand{\du}{{\delta\!u}}
\newcommand{\deta}{{\delta\!\eta}}
\newcommand{\drho}{{\delta\!\rho}}
\newcommand{\dv}{{\delta\!v}}
\newcommand{\dsigma}{{\delta\!\sigma}}
\newcommand{\dtau}{{\delta\!\tau}}

\newcommand{\wtilde}{\widetilde}
\newcommand{\<}{\langle{}}
\renewcommand{\>}{\rangle}

\newcommand{\dual}[2]{\<#1\hspace*{.5mm},#2\>}
\newcommand{\vdual}[2]{(#1\hspace*{.5mm},#2)}

\DeclareMathOperator{\dev}{dev}
\DeclareMathOperator{\tr}{tr}
\DeclareMathOperator{\id}{\mathrm{id}}
\newcommand{\CC}{\mathbb{C}}
\renewcommand{\AA}{\mathbb{A}}

\newcommand{\grad}{\nabla}

\DeclareMathOperator{\curl}{curl}
\DeclareMathOperator{\Curl}{Curl}
\newcommand{\strain}[1]{\varepsilon(#1)}
\newcommand{\strainref}[1]{\widehat\varepsilon(#1)}
\DeclareMathOperator{\rot}{rot}
\renewcommand{\div}{\operatorname{div}}
\newcommand{\dDiv}{\operatorname{div div}}
\newcommand{\divref}{\widehat{\operatorname{div}}}

\newcommand{\PiolaK}[1]{\mathcal{PK}_{{#1}}}
\newcommand{\PiolaKc}[1]{\mathcal{PK}^\perp_{{#1}}}

\newcommand{\PiolaCurl}[1]{\mathcal{P}^\mathrm{curl}_{#1}}
\newcommand{\Inn}[1]{\mathcal{I}^{nn}_{#1}}
\newcommand{\XX}[1]{P^{2,1}(#1;\SS)}

\newcommand{\ctdd}{{C_\mathrm{tdd}}}
\newcommand{\ctddinv}{{C_\mathrm{tdd}^{-1}}}
\newcommand{\sreg}{{s'}}
\newcommand{\creg}{{c'}} 

\newcommand{\mesh}{\mathcal{T}}
\newcommand{\el}{T}
\newcommand{\elref}{{\widehat T}}

\newcommand{\trnn}{{\gamma_{1,\mathit{nn}}}}

\newcommand{\G}[1]{{\Gamma_\mathit{#1}}}
\newcommand{\g}[1]{{g_{#1}}}

\title{Normal-normal continuous symmetric stresses\\ in mixed finite element elasticity
\thanks{Supported by ANID-Chile through FONDECYT project 1230013}
\author{
Carsten Carstensen\thanks{
Department of Mathematics, Humboldt--Universit\"at zu Berlin,
Unter den Linden 6, 10099 Berlin, Germany,
email: {\tt cc@math.hu-berlin.de}}
\and
Norbert Heuer\thanks{
Facultad de Matem\'aticas, Pontificia Universidad Cat\'olica de Chile,
Avenida Vicu\~na Mackenna 4860, Santiago, Chile,
email: {\tt nheuer@mat.uc.cl}}}}

\begin{document}
\maketitle
\begin{abstract}
The classical continuous mixed formulation of linear elasticity with pointwise 
symmetric stresses allows for a conforming finite element discretization 
with piecewise polynomials of degree at least three. Symmetric stress approximations
of lower polynomial order are only possible when their $\div$-conformity is
weakened to the continuity 
of normal-normal components. In two dimensions, this condition is meant pointwise
along edges for piecewise polynomials,
but a corresponding characterization for general piecewise $H(\div)$ tensors
has been elusive.

We introduce such a space and establish a continuous mixed formulation of
linear planar elasticity with pointwise symmetric stresses that have, in a distributional
sense, continuous normal-normal components across the edges of a shape-regular triangulation.
The displacement is split into an $L_2$ field and a tangential trace on the skeleton
of the mesh. The well-posedness of the new mixed formulation follows with a 
duality lemma relating the normal-normal continuous stresses with the tangential traces
of displacements.

For this new formulation we present a lowest-order conforming discretization.
Stresses are approximated by piecewise quadratic symmetric tensors, whereas
displacements are discretized by piecewise linear polynomials.
The tangential displacement trace acts as a Lagrange multiplier and
guarantees global $\div$-conformity in the limit as the mesh-size tends to zero.
We prove locking-free, quasi-optimal convergence of our scheme
and illustrate this with numerical examples.

\end{abstract}

\section{Introduction}

{\bf Motivation.}
Mixed formulations in linear elasticity do not suffer immediately from jumping materials 
and are typically robust in the incompressible limit, when the Lam\'e equations formally
become the Stokes equations. The related mathematical formulation on the continuous level
is classical in the Sobolev space $H(\div,\Omega;\SS)$ of pointwise symmetric and
globally $\div$-conforming stresses. 
In contrast, the design of pointwise symmetric, piecewise \emph{polynomial} stresses in
$H(\div,\Omega;\SS)$ is a rather delicate issue that started with Arnold and Winther
\cite{ArnoldW_02_MFE} two decades ago,
see also \cite{AdamsC_05_MFE,ArnoldAW_08_FES,Hu_15_FEA,HuZ_16_FEA}.
The lowest-order conforming element on a triangle has 24 degrees of freedom.
These degrees of freedom, as well as those of other conforming elements we are aware of,
include non-physical values at vertices: these vertex degrees of freedom are
not well defined even in $H^1(\Omega;\SS)\subset H(\div,\Omega;\SS)$.
In particular, vertex values of stresses are usually not available
from data in engineering applications to prescribe boundary conditions.
Piecewise quadratic conforming examples in $H(\div,\Omega;\SS)$ are not known and
no longer expected to exist. 
Remedies are to relax the pointwise symmetry constraint or
the global $\div$-conformity. 
This paper contributes to the second category and is motivated by the seminal schemes
by Pechstein and Sch\"oberl \cite{PechsteinS_11_TDN,PechsteinS_12_AMF,PechsteinS_18_ATM}
with continuous normal-normal components of pointwise symmetric stress approximations.
The stability analysis in \cite{PechsteinS_11_TDN}
is based on mesh-depending discrete norms, while \cite{PechsteinS_18_ATM}
suggests a pairing within the Sobolev spaces $H(\dDiv)$ and $H(\curl)$.
Without going into the details, we note that $H(\dDiv)$ is a natural space for bending moments
of Kirchhoff--Love plates and leads to conformity requirements that are quite different
from those of $H(\div,\Omega;\SS)$, cf.~\cite{RafetsederZ_18_DRK,FuehrerHN_19_UFK}.
Indeed, the idea of using normal-normal continuous tensors goes back to the
Hellan--Herrmann--Johnson method for plate bending with normal-normal continuous bending moments
\cite{Hellan_67_AEP,Herrmann_67_FEB,Johnson_73_CMF}.
The first results of Pechstein and Sch\"oberl \cite{PechsteinS_11_TDN}
appear too coarse and have no counterpart at the continuous level. 
The second, in \cite{PechsteinS_18_ATM}, lead to a  more detailed understanding
of the spaces of continuous normal-normal components of pointwise symmetric stresses.
But the resulting discrete schemes are \emph{not} conforming in $H(\div\div)$,
referred to as ``slightly nonconforming'' in \cite{PechsteinS_18_ATM}.
The finite elements in \cite{PechsteinS_11_TDN,PechsteinS_18_ATM}
are {\em not} locking free in the incompressible limit,
and require additional divergence terms in order to be so,
cf.~the first numerical example in \cite{PechsteinS_11_TDN}.
On the other hand, according to \cite{PechsteinS_12_AMF} the Pechstein--Sch\"oberl
schemes behave well for thin structures.
Our method has the same discrete conformity as the Pechstein--Sch\"oberl
elements: symmetric normal-normal continuous stresses and (traces of) displacements with continuous
tangential components. Whereas the Pechstein--Sch\"oberl elements are analyzed by using
low-regular Sobolev spaces ($H(\dDiv)$ and $H(\curl)$) that bear little mechanical relevance,
and thus allow for lowest-order approximations, our elements are analyzed within
energy spaces, that is, $H(\div)$ and traces of $H^1$, and require more degrees of freedom.
Correspondingly, our approximations are controlled in stronger norms compared with
the Pechstein--Sch\"oberl setting that uses discrete norms of lower regularity.

{\bf Contributions.}
This paper aims at the fundamental understanding of pointwise symmetric,
piecewise $H(\div)$ stresses, written as $H(\div,\mesh;\SS)$ with respect to an
underlying shape-regular triangulation $\mesh$.
Tensors of $H(\div,\mesh;\SS)$ have distributional normal traces
and we give a meaning to their normal-normal continuity.
This leads to a closed subspace $H_{nn}(\div,\mesh;\SS)\subset H(\div,\mesh;\SS)$.
The global $\div$-conformity is implemented weakly by the tangential traces of
displacements and generate a Lagrange multiplier.
The difference between $H_{nn}(\div,\mesh;\SS)$ and the globally
$\div$-conforming space $H(\div,\Omega;\SS)$ consists of tensors with non-zero
tangential-normal jumps across the skeleton $\cS$ of $\mesh$.
Therefore, the Lagrange multiplier space, denoted by $H^{1/2}_t(\cS;\R^2)$,
lives on the skeleton. Naturally, these spaces and their norms are mesh dependent.
But it is critical that stability and other constants appearing in the analysis
are mesh \emph{independent}.
This is achieved by using techniques that have been developed to analyze
discontinuous Petrov--Galerkin (DPG) schemes with optimal test functions.
The latter framework, introduced by Demkowicz and Gopalakrishnan \cite{DemkowiczG_11_CDP},
usually employs ultraweak formulations that give rise to independent trace variables,
cf.~\cite{CessenatD_98_AUW,BottassoMS_02_DPG}.
Indeed, our setting with normal-normal continuous stresses and tangential
traces of displacements can be interpreted as a mixed formulation that is ultraweak
with respect to displacements.
It is therefore natural for our scheme that displacements
be represented both by an $L_2$ field variable and an independent (tangential) trace.
The mentioned ``DPG-techniques'' refer to the canonical definition of product spaces and
related trace operators, the duality between them relating jumps with traces,
and a splitting technique to deduce inf-sup stability,
cf.~\cite{CarstensenDG_16_BSF,FuehrerHN_19_UFK} for details.
We note that our continuous mixed formulation extends to three space dimensions.
But in this paper we focus on the planar case, both for the continuous and discrete settings,
cf.~Remark~\ref{rem_3d} below.

Having our new well-posed variational formulation at hand,
we propose a lowest-order approximation on triangular meshes
that consists of piecewise quadratic symmetric stress tensors
with continuous linear normal-normal traces, discontinuous piecewise linear displacements and
tangential traces of continuous piecewise linear displacements.
There is a corresponding bounded interpolation
operator for stresses in $H(\div,\mesh;\SS)\cap L_q(\Omega;\SS)$ ($q>2$).
Our stress element has $15$ \emph{physical} degrees of freedom,
and it is crucial that they make the interpolation operator
commute with the $L_2$-projection onto piecewise linear polynomials.
The degrees of freedom can be reduced to two moments per edge by static condensation.
Condensing also the displacement field variable, this amounts to a stable
mixed pair of six stress and six displacement degrees of freedom.
Disregarding boundary conditions, this gives $2|\cE|+2|\cV|$ total degrees of freedom where
$\cE$ and $\cV$ are the sets of edges and vertices of cardinality $|\cE|$ and $|\cV|$, respectively.
Our scheme is locking free and converges quasi-optimally
in the piecewise $H(\div)$ norm for the stress,
the $L_2$ norm for the displacement field variable,
and the natural trace norm of $H^1$ for the displacement trace variable.

We claim that both the continuous setting and discretization extend to three space dimensions.
Precise definitions and proofs require additional techniques.
This is left to future research.
\footnote{Note after publication: 3d case treated in \url{http://arxiv.org/abs/2502.00609}.}

{\bf Model problem.}
We consider a bounded polygonal Lipschitz domain $\Omega\subset\R^2$
and decompose its boundary $\Gamma:=\partial\Omega$ into non-intersecting relatively open pieces
\begin{align} \label{split}
   \Gamma=\mathrm{closure}(\G{hc}\cup\G{sc}\cup\G{ss}\cup\G{sf}),
\end{align}
assigning boundary conditions
of hard clamped (``$\mathit{hc}$''), soft clamped (``$\mathit{sc}$''),
simply supported (``$\mathit{ss}$'') and traction (``$\mathit{sf}$'' for stress free) types.
Individual sets are either empty or are unions of connected sets
of positive measure, subject to the validity of Korn's inequality.
We note that an extension of our formulation and method to piecewise smooth
curved boundaries is possible. The analysis of the variational formulation requires
minor changes but the discrete analysis becomes more technical due to the
appearing non-linear element mappings.

The model problem of plane linear elasticity reads
\begin{subequations} \label{model}
\begin{align} \label{PDE}
   \sigma=2\mu\strain{u}+\lambda(\div u)\id
   \quad\text{and}\quad -\div\sigma=f\quad\text{in}\ \Omega
\end{align}
with boundary conditions
\begin{equation} \label{BC}
\begin{alignedat}{4}
   u&=\g{D} &&\text{on}\ \G{hc},\qquad
   &&u\cdot\bn=\g{D,n},\ \bt\cdot\sigma\bn=\g{N,t}\quad &&\text{on}\ \G{sc},
   \\
   u\cdot\bt&=\g{D,t},\ \bn\cdot\sigma\bn=\g{N,n}\quad &&\text{on}\ \G{ss},
   &&\sigma\bn=\g{N} &&\text{on}\ \G{sf}.
\end{alignedat}
\end{equation}
\end{subequations}
Here, $\strain{u}$ and $\id$ are the strain and identity tensors, respectively,
and $\mu,\lambda>0$ are the Lam\'e constants, related to Young's modulus $E>0$ and Poisson
ratio $\nu>0$ by
\[
   \mu = \frac E{2(1+\nu)}\quad\text{and}\quad \lambda=\frac{E\nu}{(1+\nu)(1-2\nu)}.
\]
For ease of presentation, we assume throughout that $E$ is fixed but consider a
general Poisson ratio $\nu<1/2$.
For short, we write $\sigma=\CC\strain{u}$ with elasticity tensor $\CC$, and
let $\AA:=\CC^{-1}$ be the compliance tensor. Furthermore,
$\bn$ and $\bt$ are the exterior unit normal and tangential vectors on $\Gamma$, and
we assume that data are sufficiently regular and consistent for the existence of a solution
$u\in H^1(\Omega;\R^2)$ and $\sigma\in H(\div,\Omega;\R^{2\times 2})$.
In particular, $f,\g{D},\g{N}$ are vector-valued functions, whereas
$\g{D,n},\g{D,t},\g{N,n},\g{N,t}$ are scalar functions.

{\bf Outline.} 
In the following section we introduce the
required spaces and norms, in particular operators, spaces and norms related with traces.
We define stress fields with continuous normal-normal components via duality with
$H^1$ vector fields that have vanishing tangential traces.
Such restrictions of $H^1$ vector fields are by no means canonical and require some technicalities.
We introduce the canonical $H^1$ vector trace operator $\gamma$, and the
tangential and normal component traces $\gamma_t$ and $\gamma_n$.
We stress that $\gamma_t,\gamma_n$ are only used to define special trace spaces and
the normal-normal continuity of stresses. Those spaces are analyzed in \S\ref{sec_nn}.
In \S\ref{sec_trace_nn} we invoke the space $H_{nn}(\div,\mesh;\SS)$
of normal-normal continuous stresses to define an $H^1$ vector trace $\trnn{}$ by duality.
Since (on edges not lying on the boundary) this trace operator is induced by
the tangential-normal jumps of stresses, it gives rise to the tangential component of $\gamma$.
On the boundary, $\trnn{}$ gives the full trace of $H^1$ vector functions.
We note that trace operators $\gamma$ and $\trnn{}$ are identical when
considered in duality with normal-normal continuous stress fields.
The analysis of $\trnn{}$ in \S\ref{sec_trace_nn} is in the spirit
of DPG techniques presented in \cite{CarstensenDG_16_BSF,FuehrerHN_19_UFK}.
In \S\ref{sec_mixed} we then invoke trace operator $\trnn{}$ to derive
a mixed formulation and prove its well-posedness in Theorem~\ref{thm_VF}.
The lowest-order normal-normal stress element and corresponding piecewise
polynomial approximation space and interpolation operator are presented
and analyzed in \S\ref{sec_approx}.
The resulting mixed scheme is studied in \S\ref{sec_fem}, with
locking-free quasi-optimal convergence in Theorem~\ref{thm_Cea} and
an a priori error estimate in Theorem~\ref{thm_FEM}.
In \S\ref{sec_num} we present numerical examples for smooth and singular
solutions and different boundary conditions.
They confirm the locking-free convergence with optimal order of our scheme.

{\bf Notation.}
We use the standard differential operators $\grad$ and $\div$. Furthermore,
$\curl \varphi:=(\grad \varphi)^\perp$ where
$v^\perp :=(v_2, -v_1)^\T$ for vector functions $v=(v_1,v_2)^\T$,
and $\rot v:=\div v^\perp$.
For vector functions $v$ and tensor functions $\tau$,
$\grad v$ and $\div\tau$ are calculated component and row-wise, respectively,
strain tensor $\strain{v}:=(\grad v+\grad v^\T)/2$ is the symmetric gradient of $v$,
and we set $\SS:=\{A\in\R^{2\times 2};\;A^\T=A\}$.
For a tensor $\tau$ we denote by $\dev\tau:=\tau - \tr(\tau) \id/2$
its deviatoric part with trace $\tr(\tau):=\tau :\id$ (``$:$'' is the Frobenius product).

Throughout, $a\lesssim b$ means that $a\le Cb$ with a generic constant $C>0$
that is independent of involved functions and mesh $\mesh$.
Notation $a\gtrsim b$ means that $b\lesssim a$,
and $a\simeq b$ is equivalent to $a\lesssim b$ and $a\gtrsim b$.

\section{Spaces and operators} \label{sec_ops}

Given any subdomain $\omega$ of $\Rt$ or a line segment,
we use the canonical Lebesgue and Sobolev spaces $L_q(\omega)$ and $H^s(\omega)$
($q\ge 1$, $0<s\le 2$) of scalar functions,
and need the corresponding spaces $L_q(\omega;U)$, $H^s(\omega;U)$ of functions
with values in $U\in\{\R,\Rt,\SS\}$. The $L_2(\omega;U)$ duality and norm
will be generically denoted by $\vdual{\cdot}{\cdot}_\omega$ and $\|\cdot\|_\omega$.
Occasionally, we also need boundary dualities on subsets $\partial\omega'\subset\partial\omega$
for $\omega\subset\Rt$.
These dualities are denoted by $\dual{\cdot}{\cdot}_{\partial\omega'}$ and defined with
$L_2(\partial\omega';U)$ as pivot space.
We use the space
\begin{align*}
   H^{s,r}(\div,\omega;\SS)
   &:= \{\tau\in H^s(\omega;\SS);\; \div\tau\in H^r(\omega;\Rt)\}\quad
   \text{for}\ 0\le s\le 1,\; 0\le r\le 2,
\end{align*}
and write $H(\div,\omega;\SS):=H^{0,0}(\div,\omega;\SS)$.
For $0<s<1$, Sobolev spaces $H^s(\omega;U)$ are normed with Sobolev--Slobodeckij (semi-)norms
$|\cdot|_{s,\omega}$ and $\|\cdot\|_{s,\omega}$, with canonical extension
to orders $1<s<2$. We use the same notation for $s\in\{1,2\}$,
then using the standard norm and semi-norm, except for $U=\Rt$, in which case
$|\cdot|_{1,\omega}:=\|\strain{\,\cdot\,}\|_\omega$.
Furthermore, $H^1_0(\omega;U)$ is the space of $H^1(\omega;U)$ functions with homogeneous
trace on the boundary $\partial\omega$, and $H^{-1}(\omega;U):=H^1_0(\omega;U)^*$
is the topological dual with canonical operator norm $\|\cdot\|_{-1,\omega}$.
We identify $L_2(\omega;U)=H^0(\omega;U)$,
and generally drop index $\omega$ when $\omega=\Omega$.

We consider a regular (family of) mesh(es) $\mesh$
consisting of shape-regular triangles $\el$ covering $\Omega$,
$\cup_{\el\in\mesh}\overline{\el}=\overline{\Omega}$.
The set of (open) edges of $\mesh$ is denoted by $\cE$,
and $\cE(\Omega):=\{e\in\cE;\; e\subset\Omega\}$.
The set of edges of $\el\in\mesh$ is $\cE(\el)$.
We also need the exterior unit normal and tangential vectors,
defined almost everywhere on $\partial\el$, generically denoted by $\bn$ and $\bt$,
respectively. For every edge $e\in\cE$ we select unique unit normal and tangential vectors
$\bn_e,\bt_e$ along $e$ that coincide with the exterior normal and tangential vectors,
respectively, along $\partial\Omega$ if $e\subset\partial\Omega$.

Mesh $\mesh$ induces product spaces denoted as before, but replacing
$\omega$ with $\mesh$, e.g., $H^1(\mesh):=\Pi_{\el\in\mesh} H^1(\el)$.
We will identify elements of product spaces with piecewise defined functions
on $\Omega$, e.g., $v\in H^1(\Omega)$ is identified with $v\in H^1(\mesh)$, and
$v\in H^1(\mesh)$ corresponds to an element $v\in L_2(\Omega)$.
Furthermore, $\div_\mesh$ refers to the piecewise divergence operator.

\subsection{Trace spaces} \label{sec_trace}

Product spaces and Green's formula give rise to trace operators
with support on the skeleton $\cS:=\cup\{\partial\el; \el\in\mesh\}$.
Specifically, we introduce the canonical trace operator
\begin{align*}
   \gamma:\;&\left\{\begin{array}{cll}
               H^1(\Omega;\Rt) &\rightarrow& H(\div,\mesh;\SS)^*,\\
               v &\mapsto& \dual{\gamma(v)}{\tau}_\cS
               := \vdual{\strain{v}}{\tau} + \vdual{v}{\div\tau}_\mesh,
            \end{array}\right.
\end{align*}
and
\begin{align*}
   \gamma_t:\;&\left\{\begin{array}{cll}
               H^1(\Omega;\Rt) &\rightarrow& H^1(\mesh)^*,\\
               v &\mapsto& \dual{\gamma_t(v)}{z}_\cS
               := \vdual{\rot v}{z} - \vdual{v}{\curl z}_\mesh
            \end{array}\right.
   \quad\text{(tangential trace)},\\[1em]
   \gamma_n:\;&\left\{\begin{array}{cll}
               H^1(\Omega;\Rt) &\rightarrow& H^1(\mesh)^*,\\
               v &\mapsto& \dual{\gamma_n(v)}{z}_\cS
               := \vdual{\div v}{z} + \vdual{v}{\grad z}_\mesh
            \end{array}\right.
   \qquad\text{(normal trace)}.
\end{align*}
Here, $\vdual{\cdot}{\cdot}_\mesh$ is the generic notation for
$L_2(\mesh;U)$ dualities, $U\in\{\R,\Rt,\SS\}$.

The kernels of $\gamma_t$ and $\gamma_n$
\[
   H^1_n(\Omega,\cS;\Rt):=\mathrm{ker}(\gamma_t),\quad
   H^1_t(\Omega,\cS;\Rt):=\mathrm{ker}(\gamma_n)
\]
are closed subspaces of $H^1(\Omega;\Rt)$.
We therefore have the following trace spaces,
\begin{align*}
      H^{1/2}(\cS;\Rt)       &:= \gamma(H^1(\Omega;\Rt)),
   && H^{1/2}_0(\cS;\Rt)      := \gamma(H^1_0(\Omega;\Rt)),\\
      \wtilde H^{1/2}_n(\cS;\Rt)     &:= \gamma(H^1_n(\Omega,\cS;\Rt)),
   && \wtilde H^{1/2}_{0,n}(\cS;\Rt)  := \gamma(H^1_n(\Omega,\cS;\Rt)\cap H^1_0(\Omega;\Rt)),\\
      \wtilde H^{1/2}_t(\cS;\Rt)     &:= \gamma(H^1_t(\Omega,\cS;\Rt)),
   && \wtilde H^{1/2}_{0,t}(\cS;\Rt)  := \gamma(H^1_t(\Omega,\cS;\Rt)\cap H^1_0(\Omega;\Rt)).
\end{align*}
Spaces $\wtilde H^{1/2}_n(\cS;\Rt)$, $\wtilde H^{1/2}_{0,n}(\cS;\Rt)$,
$\wtilde H^{1/2}_t(\cS;\Rt)$, and $\wtilde H^{1/2}_{0,t}(\cS;\Rt)$
are merely needed to define and analyze the normal-normal continuity of tensors.

We note that, due to the definition of the trace operators with product test spaces,
traces are elements of product spaces, e.g.,
$H^{1/2}(\cS;\Rt)\subset \Pi_{\el\in\mesh} H^{1/2}(\partial\el;\Rt)$ with obvious definition
of the local spaces $H^{1/2}(\partial\el;\Rt)$.
Since elements of $H^{1/2}(\cS;\Rt)$ are traces of $H^1(\Omega;\Rt)$ functions
their components on edges $e\in\cE$ are univariate.
We can therefore identify functions
$\varphi=(\varphi_\el)_{\el\in\mesh}\in H^{1/2}(\cS;\Rt)$ with their components on edges,
$\varphi=(\varphi_e)_{e\in\cE}\in H^{1/2}(\cE;\Rt):=\Pi_{e\in\cE} H^{1/2}(e;\Rt)$.

Let us introduce some further notation.
For a function $v\in H^1(\el;U)$ ($\el\in\mesh$, $U\in\{\R,\Rt\}$)
we need its trace on the boundary $\partial\el$ and on edges $e\in\cE(\el)$.
They will simply be denoted by $v|_{\partial\el}\in H^{1/2}(\partial\el;U)$
and $v|_e\in H^{1/2}(e;U)$.
We also need the trace space $\wtilde H^{1/2}(e)$ which can be characterized as
the restriction of functions $v\in H^{1/2}(\partial T)$
with $v|_{e'}=0$ ($e\not=e'\in\cE(\el)$). For details we refer to \cite{McLean_00_SES}.
We denote $\wtilde H^{1/2}(\cE):=\Pi_{e\in\cE} \wtilde H^{1/2}(e)$.

There is an inherent duality between $H^{1/2}(\cS;\Rt)$ and $H(\div,\mesh;\SS)$,
\[
   \dual{\varphi}{\tau}_\cS := \dual{\gamma(v)}{\tau}_\cS
   \quad (\varphi\in H^{1/2}(\cS;\Rt), \tau\in H(\div,\mesh;\SS))
\]
with any element $v\in H^1(\Omega;\Rt)$ such that $\gamma(v)=\varphi$.
Using this duality, we finally introduce the space of pointwise symmetric stresses
with continuous normal-normal components across $\cS$ as
\[
   H_{nn}(\div,\mesh;\SS) :=
   \{\tau\in H(\div,\mesh;\SS);\;
     \dual{\varphi}{\tau}_\cS=0\ \forall \varphi\in \wtilde H^{1/2}_{0,n}(\cS;\Rt)\}.
\]

\subsection{Analysis of spaces $\wtilde H^{1/2}_n(\cS;\Rt)$, $\wtilde H^{1/2}_t(\cS;\Rt)$, and
$H_{nn}(\div,\mesh;\SS)$.} \label{sec_nn}

We characterize the elements of trace spaces $\wtilde H^{1/2}_t(\cS;\Rt)$ and $\wtilde H^{1/2}_n(\cS;\Rt)$.
It is well known that the continuity of functions of $H^{1/2}$ on curves cannot be imposed locally
(at points). This is related to the fact that the order $1/2$ of Sobolev spaces on
one-dimensional manifolds is the limit case in which point evaluations are not well defined.
We therefore have to use the spaces $\wtilde H^{1/2}(e)$ ($e\in\cE$) introduced before. They reflect a
certain decaying condition of functions at the endpoints of $e$, this being equivalent to
our definition which is the continuity in $H^{1/2}$ of extension by zero onto Lipschitz
curves containing $e$.
For the next statement, recall the identification of
$\varphi=(\varphi_\el)_{\el\in\mesh}\in H^{1/2}(\cS;\Rt)$ with elements of the product space
$\varphi=(\varphi_e)_{e\in\cE}\in H^{1/2}(\cE;\Rt)$. For such functions we will abbreviate
$\varphi\cdot\bn|_\cE:=(\varphi_e\cdot\bn_e)_{e\in\cE}$ and
$\varphi\cdot\bt|_\cE:=(\varphi_e\cdot\bt_e)_{e\in\cE}$.

\begin{prop} \label{prop_trace}
Any function $\varphi=(\varphi_e)_{e\in\cE}\in H^{1/2}(\cS;\Rt)$ satisfies
\begin{align*}
   &\varphi\in \wtilde H^{1/2}_t(\cS;\Rt) \quad\Leftrightarrow\quad
   \varphi\cdot\bn|_\cE=0\quad\text{and}\quad
   \varphi\cdot\bt|_\cE\in \wtilde H^{1/2}(\cE)\\
\intertext{and}
   &\varphi\in \wtilde H^{1/2}_n(\cS;\Rt) \quad\Leftrightarrow\quad
   \varphi\cdot\bt|_\cE=0\quad\text{and}\quad
   \varphi\cdot\bn|_\cE\in \wtilde H^{1/2}(\cE).
\end{align*}
\end{prop}

\begin{proof}
We prove the first statement. The second follows analogously.
To this end, let $\varphi=\gamma(v)\in H^{1/2}(\cS;\Rt)$ be given with $v\in H^1(\Omega;\Rt)$.
For an element $\el\in\mesh$ let $x_j$ denote its vertices with opposite edges $e_j$ ($j=1,2,3$).
The corresponding barycentric coordinates and tangential (resp. normal) vectors are
$\lambda_j$ and $\bt_j$ (resp. $\bn_j$). Assuming that $v\cdot\bn_j|_{e_j}=0$ ($j=1,2,3$)
we have to show that $v\cdot\bt_j|_{e_j}\in\wtilde H^{1/2}(e_j)$ ($j=1,2,3$).

It is enough to show that $v\cdot\bt_1|_{e_1}\in\wtilde H^{1/2}(e_1)$.
There are constants $c_2,c_3,\alpha_2,\alpha_3\in \R$ with $\alpha_2,\alpha_3\not=0$ such that
\[
   \bt_1+c_2\bn_1=\alpha_2\bn_2\quad\text{and}\quad \bt_1+c_3\bn_1=\alpha_3\bn_3.
\]
In fact, $\{\bt_1,\bn_1\}$ being a basis of $\Rt$, there are numbers $\xi_j,\eta_j\in\R$ ($j=2,3$)
such that $\xi_j\bt_1+\eta_j\bn_1=\bn_j$ ($j=2,3$). In particular, $\xi_j\not=0$ since, otherwise,
$\bn_1=\bn_j$ ($j\in\{2,3\}$). It follows that $c_j=\eta_j/\xi_j$ and $\alpha_j=1/\xi_j$ ($j=2,3$).
We continue to select $g:=(\bt_1+c_2\bn_1)\lambda_3+(\bt_1+c_3\bn_1)\lambda_2$ and define
$w:=g\cdot v|_\el\in H^1(\el)$. Since by assumption $v\cdot\bn_j|_{e_j}=0$ ($j=1,2,3$),
$w$ satisfies
\begin{align*}
   w|_{e_1}= \lambda_3(\bt_1+c_2\bn_1)\cdot v|_{e_1}+\lambda_2(\bt_1+c_3\bn_1)\cdot v|_{e_1}=
             (\lambda_3+\lambda_2)\bt_1\cdot v|_{e_1}=v\cdot\bt_1|_{e_1}
\end{align*}
and
\begin{align*}
   w|_{e_2}=\lambda_3(\bt_1+c_2\bn_1)\cdot v|_{e_2}=\alpha_2\lambda_3v\cdot\bn_2|_{e_2}=0,\quad
   w|_{e_3}=\lambda_2(\bt_1+c_3\bn_1)\cdot v|_{e_3}=\alpha_3\lambda_2v\cdot\bn_3|_{e_3}=0.
\end{align*}
In other words, $w|_{\partial\el}\in H^{1/2}(\partial\el)$
is a zero extension of $v\cdot\bt_1|_{e_1}$ to $e_2$ and $e_3$,
that is, $\varphi\cdot\bt_1|_{e_1}=v\cdot\bt_1|_{e_1}\in\wtilde H^{1/2}(e_1)$.
\end{proof}

\begin{prop} \label{prop_cont}
The inclusion
\begin{align} \label{dense}
   \wtilde H^{1/2}_{0,n}(\cS;\Rt)\oplus\wtilde H^{1/2}_{0,t}(\cS;\Rt) &\subset H^{1/2}_0(\cS;\Rt)
\end{align}
is dense. If there exists (at least) one edge $e\in\cE(\Omega)$ that does not touch the boundary,
$\mathrm{dist}(e,\G{})>0$, then this inclusion is strict.
Furthermore, any $\tau\in H(\div,\mesh;\SS)$ satisfies
\begin{align} \label{div_cont}
   \tau\in H(\div,\Omega;\SS)
   \quad\Leftrightarrow\quad
   \dual{\varphi}{\tau}_\cS = 0\quad
   \forall \varphi\in \wtilde H^{1/2}_{0,n}(\cS;\Rt)\oplus\wtilde H^{1/2}_{0,t}(\cS;\Rt).
\end{align}
\end{prop}

\begin{proof}
We start by proving \eqref{dense}. The sum is obviously direct for geometric reasons.
If there is no interior edge, $\cE(\Omega)=\emptyset$, then
there is nothing to show since $\mesh=\{\Omega\}$ and
$H^{1/2}_0(\cS;\Rt)=H^{1/2}_0(\Gamma;\Rt)=\{0\}$ in that case.
Let $\cE(\Omega)\not=\emptyset$.
The inclusion
$\wtilde H^{1/2}_{0,n}(\cS;\Rt)\oplus\wtilde H^{1/2}_{0,t}(\cS;\Rt)\subset H^{1/2}_0(\cS;\Rt)$
holds by definition of the spaces.
To show density we use the density in $H^1_0(\Omega;\Rt)$ of the $C^\infty_0(\Omega;\Rt)$
functions that vanish in a neighborhood of every interior vertex of $\mesh$, 
cf.~\cite[Lemma~17.3]{Tartar_07_ISS}, and the corresponding density of traces in
$H^{1/2}_0(\cS;\Rt)$. Given such a function $\phi\in C^\infty_0(\Omega;\Rt)$ (vanishing in
a neighborhood of interior vertices), a partition of unity allows to find a representation
$\phi=\sum_{e\in\cE(\Omega)}\phi_e$ such that $\phi_e\in C^\infty_0(\omega_e)$.
Here, $\omega_e$ is the patch of the two elements that have edge $e$ in common,
$\omega_e=\mathrm{int}(\overline{\el}_1\cup\overline{\el}_2)$, 
$\el_1,\el_2\in\mesh$, $\el_1\not=\el_2$, $e\in\cE(\el_1)\cap\cE(\el_2)$.
Denoting by $\bn_e$ and $\bt_e$ the normal and tangential vectors on $e$
(in a certain orientation), we split
$\phi_e=\phi_{e,n}+\phi_{e,t}:=
(\phi_e\cdot\bn_e)\bn_e+(\phi_e\cdot\bt_e)\bt_e$ ($e\in\cE(\Omega)$).
One observes that
\[
   \phi_n:=\sum_{e\in\cE(\Omega)}\phi_{e,n}\in H^1_n(\Omega,\cS;\Rt)\cap H^1_0(\Omega;\Rt),\quad
   \phi_t:=\sum_{e\in\cE(\Omega)}\phi_{e,t}\in H^1_t(\Omega,\cS;\Rt)\cap H^1_0(\Omega;\Rt),
\]
and $\phi=\phi_n+\phi_t$.
Therefore, $\varphi_n:=\gamma(\phi_n)$ and $\varphi_t:=\gamma(\phi_t)$
satisfy $\varphi_n\in \wtilde H^{1/2}_{0,n}(\cS;\Rt)$,
$\varphi_t\in \wtilde H^{1/2}_{0,t}(\cS;\Rt)$, and $\varphi_n+\varphi_t=\gamma(\phi)$.
This proves the density of relation \eqref{dense}.
The strictness can be reduced to the fact that non-zero constants are not elements of
$\wtilde H^{1/2}(e)$, cf.~\cite[Theorem~1.5.2.8]{Grisvard_85_EPN}, but constant vector functions
are certainly $H^1$ regular and have constant traces.

Direction ``$\Rightarrow$'' in \eqref{div_cont} holds because $\dual{\gamma(v)}{\tau}_\cS=0$
is true for any $v\in H^1_0(\Omega;\Rt)$ by Green's formula, and since
$\wtilde H^{1/2}_{0,n}(\cS;\Rt)\oplus\wtilde H^{1/2}_{0,t}(\cS;\Rt)\subset H^{1/2}_0(\cS;\Rt)$.
The other direction ``$\Leftarrow$'' in \eqref{div_cont} is due to the density of
relation \eqref{dense} and integration by parts.
More precisely, given $\tau\in H(\div,\mesh;\SS)$ that satisfies
the relation on the right-hand side of \eqref{div_cont}, density of \eqref{dense}
implies that $\dual{\varphi}{\tau}_\cS=0$ for any $\varphi\in H^{1/2}_0(\cS;\Rt)$.
This provides $\vdual{v}{\div\tau}_\mesh=-\vdual{\strain{v}}{\tau}$
for any smooth vector function $v$ with compact support in $\Omega$ and so
$\div\tau\in L_2(\Omega;\Rt)$.
\end{proof}

The next corollary follows by definition of $H_{nn}(\div,\mesh;\SS)$.

\begin{cor} \label{cor_cont}
Any $\tau\in H_{nn}(\div,\mesh;\SS)$ satisfies
\[
   \tau\in H(\div,\Omega;\SS)
   \quad\Leftrightarrow\quad
   \dual{\varphi}{\tau}_\cS = 0\quad
   \forall \varphi\in \wtilde H^{1/2}_{0,t}(\cS;\Rt).
\]
\end{cor}

Even though trace space $\wtilde H^{1/2}_{0,t}(\cS;\Rt)$ is sufficient to control the jumps
of $\tau\in H_{nn}(\div,\mesh;\SS)$ (as seen in the corollary), it is
not large enough to represent elements of $H^{1/2}(\cS;\Rt)$ that are dual to
(traces of) $H_{nn}(\div,\mesh;\SS)$ elements. We therefore need another tangential trace space.

\subsection{Duality trace operator and space} \label{sec_trace_nn}

The subsequent trace operator is at the heart of our mixed formulation with
normal-normal continuous stresses. It is the restriction of the canonical trace operator
$\gamma$, defined by
\begin{align*}
   \trnn:\;&\left\{\begin{array}{cll}
               H^1(\Omega;\Rt) &\rightarrow& H_{nn}(\div,\mesh;\SS)^*,\\
               v &\mapsto& \dual{\trnn(v)}{\tau}_\cS
               := \vdual{\strain{v}}{\tau} + \vdual{v}{\div\tau}_\mesh,
            \end{array}\right.
\end{align*}
and gives rise to the trace spaces
\begin{align*}
   H^{1/2}_t(\cS;\Rt) &:= \trnn\bigl(H^1(\Omega;\Rt)\bigr),\quad\\
   H^{1/2}_{0,t}(\cS;\Rt) &:= \trnn\bigl(H^1_0(\Omega;\Rt)\bigr) \subset
   H^{1/2}_{D,t}(\cS;\Rt) := \trnn\bigl(H^1_D(\Omega;\Rt)\bigr)
\end{align*}
with the space of displacements with homogeneous Dirichlet conditions
\[
   H^1_D(\Omega;\Rt):=\{v\in H^1(\Omega;\Rt);\;
   v|_\G{hc}=0,\ v\cdot\bn|_\G{sc}=0,\ v\cdot\bt|_\G{ss}=0\}.
\]
We also need the space of stress tensors with homogeneous Neumann conditions,
\[
   H_N(\div,\Omega;\SS) := \{\tau\in H(\div,\Omega;\SS);\;
   (\bt\cdot\tau\bn)|_\G{sc}=0,\ (\bn\cdot\tau\bn)|_\G{ss}=0,\ \tau\bn|_\G{sf}=0\}.
\]
Here, the Neumann traces are restrictions of the standard normal-trace operator
$H(\div,\Omega;\SS)\ni\tau\mapsto \tau\bn|_\G{}\in H^{-1/2}(\G{};\Rt)$.

\begin{remark} \label{rem_trace}
By definition, $\dual{\trnn(v)}{\tau}_\cS=\dual{\gamma(v)}{\tau}_\cS$
for $v\in H^1(\Omega;\Rt)$ and $\tau\in H_{nn}(\div,\mesh;\SS)$.
Specifically, considering the continuity of normal-normal traces of
$\tau\in H_{nn}(\div,\mesh;\SS)$ across edges,
$\trnn(v)$ amounts to tangential traces on interior edges whereas,
on edges $e\subset\Gamma$, it corresponds to the standard vector-valued trace $v|_e$.
On the other hand, $\gamma(v)$ is the standard vector-valued trace of $v$ on the skeleton
and reduces to $\trnn(v)$ when tested with $\tau\in H_{nn}(\div,\mesh;\SS)$.
\end{remark}

We note the following relation.

\begin{lemma} \label{la_cont}
Any $\tau\in H_{nn}(\div,\mesh;\SS)$ satisfies
\[
   \tau\in H_N(\div,\Omega;\SS)
   \quad\Leftrightarrow\quad
   \dual{\rho}{\tau}_\cS = 0\quad
   \forall \rho\in H^{1/2}_{D,t}(\cS;\Rt).
\]
\end{lemma}

\begin{proof}
Direction ``$\Rightarrow$'' is straightforward as noted in the proof of 
Proposition~\ref{prop_cont} and by using Green's identity.
The converse direction ``$\Leftarrow$'' can be shown in two steps.
Let a tensor $\tau\in H_{nn}(\div,\mesh;\SS)$ be given with
$\dual{\rho}{\tau}_\cS = 0$ for any $\rho\in H^{1/2}_{D,t}(\cS;\Rt)$.
The regularity $\tau\in H(\div,\Omega;\SS)$ can be seen
directly in a distributional way. It also follows from Corollary~\ref{cor_cont} since
$\wtilde H^{1/2}_{0,t}(\cS;\Rt)\subset H^{1/2}_{0,t}(\cS;\Rt)\subset H^{1/2}_{D,t}(\cS;\Rt)$.
In fact, by definition,
any $\rho\in\wtilde H^{1/2}_{0,t}(\cS;\Rt)$ is the trace
$\gamma(v)$ of some function $v\in H^1_{0,t}(\Omega,\cS;\Rt)$.
In particular, $v\in H^1_0(\Omega;\Rt)$, and $\gamma(v)=\rho$ implies
$\trnn(v)=\rho$, that is, $\rho\in H^{1/2}_{0,t}(\cS;\Rt)$.
To see that $\tau$ has zero Neumann traces we note that any $v\in H^1_D(\Omega;\Rt)$ satisfies
\[
   \dual{\trnn(v)}{\tau}_\cS
   = \vdual{\strain{v}}{\tau}+\vdual{v}{\div\tau}
   = \dual{v}{\tau\bn}_\G{}.
\]
By appropriate selections of $v$ (with zero trace on $\G{}$ except
for $(v\cdot\bt)|_\G{sc}$, $(v\cdot\bn)|_\G{ss}$, and $v|_\G{sf}$, separately)
it is clear that the traces
$(\bt\cdot\tau\bn)|_\G{sc}$, $(\bn\cdot\tau\bn)|_\G{ss}$, and $\tau\bn|_\G{sf}$ vanish.
This finishes the proof.
\end{proof}

Let us introduce norms.
For $v\in H^1(\Omega;\Rt)$ we use
$\|v\|_1:=\Bigl(\|v\|^2 + \|\strain{v}\|^2\Bigr)^{1/2}$ and, for $\tau\in H(\div,\Omega;\SS)$,
$\|\tau\|_{\div}:=\Bigl(\|\tau\|^2 + \|\div\tau\|^2\Bigr)^{1/2}$.
For an element $\tau$ of the product space $H(\div,\mesh;\SS)$
the corresponding product norm is denoted as $\|\tau\|_{\div,\mesh}$.
Throughout, $\|\cdot\|_\mesh$ is the generic $L_2(\mesh)$ product norm.
For traces $\rho\in H^{1/2}_{t}(\cS;\Rt)$ we use the minimal extension norm
\[
   \|\rho\|_{1/2,t,\cS} :=
   \mathrm{inf}
   \{\|v\|_1;\; v\in H^1(\Omega;\Rt),\ \trnn(v)=\rho\}.
\]
It turns out that this canonical trace norm has a duality representation.

\begin{prop} \label{prop_norm}
Any $\rho\in H^{1/2}_{t}(\cS;\Rt)$ satisfies
\[
   \|\rho\|_{1/2,t,\cS} =
   \sup_{0\not=\tau\in H_{nn}(\div,\mesh;\SS)}
   \frac {\dual{\rho}{\tau}_\cS}{\|\tau\|_{\div,\mesh}}.
\]
In the case of $\G{ss}\cup\G{sf}=\emptyset$ and $\rho\in H^{1/2}_{D,t}(\cS;\Rt)$,
the same relation holds with supremum taken for $\tau\in H_{nn}(\div,\mesh;\SS)\setminus\{0\}$
with $\vdual{\tr(\tau)}{1}=0$.
\end{prop}

\begin{proof}
Given $\rho\in H^{1/2}_{t}(\cS;\Rt)$ there is $v\in H^1(\Omega;\Rt)$ with
$\trnn(v)=\rho$ so that
\[
   \dual{\rho}{\tau}_\cS = \vdual{\strain{v}}{\tau}+\vdual{v}{\div\tau}_\mesh
   \le
   \|v\|_1 \|\tau\|_{\div,\mesh} \quad\forall\tau\in H_{nn}(\Omega,\mesh;\SS).
\]
This shows the inequality ``$\ge$''.
The inequality ``$\le$'' follows with arguments from the DPG analysis in
\cite[Theorem~2.3]{CarstensenDG_16_BSF}.
We use a variational approach as in \cite[Lemma 3.2]{FuehrerHN_19_UFK},
\cite[Proposition~3]{FuehrerHN_22_DMS}.
For given $\rho\in H^{1/2}_{t}(\cS;\Rt)$ we define $\tau\in H_{nn}(\div,\mesh;\SS)$ as the solution to
\begin{align} \label{def_tau}
   \vdual{\tau}{\dtau} + \vdual{\div\tau}{\div\dtau}_\mesh
   = \dual{\rho}{\dtau}_\cS\quad\forall \dtau\in H_{nn}(\div,\mesh;\SS).
\end{align}
Thereafter, let $v\in H^1(\Omega;\Rt)$ solve
\begin{align} \label{def_v}
   \vdual{v}{\dv} + \vdual{\strain{v}}{\strain{\dv}}
   = \dual{\trnn(\dv)}{\tau}_\cS\quad\forall \dv\in H^1(\Omega;\Rt).
\end{align}
One finds that $v=\div_\mesh\tau$ (the piecewise divergence of $\tau$) and $\strain{v}=\tau$
so that, by definition of $\trnn$ and relation \eqref{def_tau},
\[
   \dual{\trnn(v)}{\dtau}_\cS
   = \vdual{\strain{v}}{\dtau} + \vdual{v}{\div\dtau}_\mesh
   = \vdual{\tau}{\dtau} + \vdual{\div\tau}{\div\dtau}_\mesh
   = \dual{\rho}{\dtau}_\cS
\]
for any $\dtau\in H_{nn}(\div,\mesh;\SS)$.
In particular, $\trnn(v)=\rho$ and $\|\rho\|_{1/2,t,\cS}\le \|v\|_1$.
The test functions $\dtau=\tau$ in \eqref{def_tau} and $\dv=v$ in \eqref{def_v}
show that
\[
   \|\tau\|_{\div,\mesh}^2 = \dual{\rho}{\tau}_\cS = \|v\|_1^2.
\]
This concludes the proof of the first statement.

Let us consider the particular case of
$\G{ss}\cup\G{sf}=\emptyset$ and $\rho\in H^{1/2}_{D,t}(\cS;\Rt)$.
Then, any $w\in H^1(\Omega;\Rt)$ with $\trnn(w)=\rho$ satisfies $w\in H^1_D(\Omega;\Rt)$
so that
\[
   \vdual{\div w}{1} = \dual{w}{\bn}_\Gamma = \dual{w}{\bn}_\G{hc}+\dual{w\cdot\bn}{1}_\G{sc} = 0.
\]
In the construction above, $\rho=\trnn(v)$ and $\tau=\strain{v}$. Therefore,
$\tau$ satisfies
\[
   \vdual{\tr(\tau)}{1} = \vdual{\strain{v}}{\id} = \vdual{\div v}{1} = 0.
\]
This shows the second statement.
\end{proof}

To consider Dirichlet boundary conditions in our mixed formulation,
we need to restrict elements $\rho\in H^{1/2}_t(\cS;\Rt)$ to subsets of $\Gamma$.
This amounts to taking traces of an element $v\in H^1(\Omega;\Rt)$
with $\trnn(v)=\rho$, giving rise to bounded linear operators
for $\mathit{bc}\in\{\mathit{hc,sc,ss}\}$
\[
   H^{1/2}_t(\cS;\Rt)\ni \rho
   \mapsto v\in  H^1(\Omega;\Rt) \mapsto
   \rho|_\G{bc}:=v|_\G{bc}\in H^{1/2}(\G{bc};\Rt).
\]
Recall that $(\cdot)|_\G{bc}:\;H^1(\Omega;\Rt)\to H^{1/2}(\G{bc};\Rt)$
denotes the standard trace operator.

We will need the following generalized trace-dev-div lemma
from \cite{CarstensenH_FOT}, here formulated in two space dimensions.

\begin{lemma} \label{la_trdevdiv}
For $0\le s\le 1$ let $\Sigma\subset H^{s}(\Omega;\Rtt)$
be a closed subspace that does not contain $\id$.
Estimate
\[
   \ctddinv \|\tr(\tau)\|_s \le \|\dev\tau\|_s + \|\div\tau\|_{s-1}
   \quad\forall\tau\in \Sigma
\]
holds true with a constant $\ctdd>0$ that solely depends on $s$, $\Omega$,
and $\Sigma$.
\end{lemma}

\section{Mixed formulation} \label{sec_mixed}

We aim at solving model problem \eqref{PDE},
$\sigma=\CC\strain{u}$, $-\div\sigma=f$ in $\Omega$,
subject to boundary conditions \eqref{BC},
\begin{alignat*}{4}
   u|_\G{hc}&=\g{D},\quad
   &(u\cdot\bn)|_\G{sc}&=\g{D,n},
   &(\bt\cdot\sigma\bn)|_\G{sc}&=\g{N,t},\\
   & &(u\cdot\bt)|_\G{ss}&=\g{D,t},\quad
   &(\bn\cdot\sigma\bn)|_\G{ss}&=\g{N,n},\quad
   &(\sigma\bn)|_\G{sf}&=\g{N},
\end{alignat*}
with boundary decomposition \eqref{split},
\(
   \Gamma=\mathrm{closure}(\G{hc}\cup\G{sc}\cup\G{ss}\cup\G{sf}).
\)
It goes without saying that the restrictions above are taken to be $H^{1/2}$ respectively
$H^{-1/2}$ traces of displacements and stresses onto the corresponding boundary pieces.
Of course, if any of the boundary pieces
comprises more than one edge of the polygon $\Gamma$, then the trace
space is of piecewise $H^{1/2}$ or $H^{-1/2}$ type.

We have to assume that the boundary data are compatible so that there exist
functions $u\in H^1(\Omega;\Rt)$, $\sigma\in H(\div,\Omega;\SS)$ that satisfy \eqref{BC}.
Furthermore, we request that the boundary conditions fix rigid body motions such that
Korn's inequality provides a constant $C(\Omega,\G{D})>0$ with
\begin{align} \label{Korn}
   \|\strain{u}\| \ge C(\Omega,\G{D}) \|u\|\quad\forall u\in H^1_D(\Omega;\Rt).
\end{align}
Here, ``$\G{D}$'' indicates that the number $C(\Omega,\G{D})$ depends on
the Dirichlet boundary parts $\G{hc}$, $\G{sc}$, $\G{ss}$.
To simplify the handling of boundary conditions we employ corresponding extensions of the data.
Throughout the remainder of this paper, these extensions will be arbitrary but fixed functions
$u_D\in H^1(\Omega;\Rt)$ and $\sigma_N\in H(\div,\Omega;\SS)$ with
\begin{subequations} \label{ext}
\begin{alignat}{4}
   u_D|_\G{hc}&=\g{D},\quad
   &(u_D\cdot\bn)|_\G{sc}&=\g{D,n},
   &(\bt\cdot\sigma_N\bn)|_\G{sc}&=\g{N,t},\\
   & &(u_D\cdot\bt)|_\G{ss}&=\g{D,t},\quad
   &(\bn\cdot\sigma_N\bn)|_\G{ss}&=\g{N,n},\quad
   &(\sigma_N\bn)|_\G{sf}&=\g{N}.
\end{alignat}
\end{subequations}
The constant tensor 
\begin{align} \label{sigma0}
   \sigma_0
   := 
   \begin{cases}
      0 \in\SS\quad \text{if}\ \G{ss}\cup\G{sf}\not=\emptyset,\ \text{and else}\\
   \frac {\lambda+\mu}{|\Omega|}
   \Bigl(\dual{\g{D}\cdot\bn}{1}_\G{hc} + \dual{\g{D,n}}{1}_\G{sc}\Bigr)\id
   \end{cases}
\end{align}
will solely play a role when $\G{ss}\cup\G{sf}=\emptyset$.

\begin{remark} \label{rem_sigma0}
In the case of $\G{ss}\cup\G{sf}=\emptyset$,
\[
   \|\sigma_0\| =
   \frac {\lambda+\mu}{|\Omega|} |\dual{u_D\cdot\bn}{1}_\G{}|\|\id\|
   =
   (\lambda+\mu) \sqrt{\frac {2}{|\Omega|}} |\vdual{\div u_D}{1}|
\]
and in \eqref{trsigma0} below we will see that
$\vdual{\tr(\sigma_0)}{1}=\vdual{\tr(\sigma)}{1}$.
\end{remark}

As indicated before, let $\mesh$ be a shape-regular refinement of an initial
regular triangulation $\mesh_0$ of $\Omega$. We require that
$\mesh$ induces a boundary mesh that is compatible with boundary decomposition \eqref{split},
that is, each of the non-empty boundary pieces is covered by a family of (closures of)
entire edges of $\mesh$.

We write problem \eqref{model} in a mixed form with $\sigma\in H_{nn}(\div,\mesh;\SS)$
and $u\in L_2(\Omega;\Rt)$ as follows.
Applying compliance tensor $\AA$ to the constitutive equation,
testing with $\dsigma\in H_{nn}(\div,\mesh;\SS)$ and applying trace operator $\trnn$, we obtain
\[
   \vdual{\AA\sigma}{\dsigma} + \vdual{u}{\div\dsigma}_\mesh - \dual{\trnn(u)}{\dsigma}_\cS=0.
\]
Then, introducing the independent trace variable $\eta=\trnn(u)$,
and weakly writing the equilibrium equation with weakly imposed conformity by Lemma~\ref{la_cont},
we arrive at the following mixed scheme.

\emph{Find $\sigma\in H_{nn}(\div,\mesh;\SS)$, $u\in L_2(\Omega;\Rt)$, and
$\eta\in H^{1/2}_{t}(\cS;\Rt)$ such that}
\begin{subequations} \label{VF}
\begin{align} \label{VFbc}
   &\eta|_\G{hc}=\g{D},\ \eta\cdot\bn|_\G{sc}=\g{D,n},\ \eta\cdot\bt|_\G{ss}=\g{D,t}
\end{align}
\emph{and, for any $\dsigma\in H_{nn}(\div,\mesh;\SS)$, $\du\in L_2(\Omega;\Rt)$, and
$\deta\in H^{1/2}_{D,t}(\cS;\Rt)$,}
\begin{alignat}{3}
   &\vdual{\AA\sigma}{\dsigma} + \vdual{u}{\div\dsigma}_\mesh - \dual{\eta}{\dsigma}_\cS
   &&= 0,\label{VFa}\\
   &\vdual{\div\sigma}{\du}_\mesh - \dual{\deta}{\sigma}_\cS
   &&= -\vdual{f}{\du}-\dual{\deta}{\sigma_N}_\cS\label{VFb}.
\end{alignat}
\end{subequations}

\begin{remark} \label{rem_bc}
Note that the selection of $\deta\in H^{1/2}_{0,t}(\cS;\Rt)$ in \eqref{VFb} ensures that
$\sigma\in H(\div,\Omega;\SS)$, whereas the enlarged space
$H^{1/2}_{D,t}(\cS;\Rt)$ of test functions $\deta$ accounts for a weak form of
the Neumann boundary conditions,
\[
  \bt\cdot\sigma\bn|_\G{sc}=\bt\cdot\sigma_N\bn|_\G{sc}=\g{N,t},\quad
  \bn\cdot\sigma\bn|_\G{ss}=\bn\cdot\sigma_N\bn|_\G{ss}=\g{N,n},\quad
  \sigma\bn|_\G{sf}=\sigma_N\bn|_\G{sf}=\g{N},
\]
cf.~Lemma~\ref{la_cont}.
%
Alternatively, the Neumann boundary conditions can be treated like essential ones.
In that case, test space $H^{1/2}_{D,t}(\cS;\Rt)$ has to be replaced with $H^{1/2}_{0,t}(\cS;\Rt)$.
However, the explicit specification of Neumann conditions requires the definition of yet
another trace operator, acting on $H_{nn}(\div,\mesh;\SS)$.
For ease of presentation we proceed with the weak form.
\end{remark}

One of our main results is the well-posedness of the new weak formulation.

\begin{theorem} \label{thm_VF}
Let $f\in L_2(\Omega;\Rt)$ and boundary data as in \eqref{ext} be given.
Furthermore, let boundary decomposition \eqref{split} be so that Korn's inequality
\eqref{Korn} holds. Then problem \eqref{VF} is well posed.
Let $(\sigma,u,\eta)$ denote its solution and recall tensor $\sigma_0$ defined in \eqref{sigma0}.
The couple $(u,\sigma)\in H^1(\Omega)\times H(\div,\Omega;\SS)$ solves \eqref{model},
$\eta=\trnn(u)$, and
\begin{align*}
   \|\sigma-\sigma_0\|_{\div,\mesh}
   + \|u\| + \|\eta\|_{1/2,t,\cS} &\lesssim \|f\| + \|u_D\|_1 + \|\sigma_N\|_{\div}
\end{align*}
holds uniformly with respect to $\nu$ and $\mesh$.

Since extensions $u_D$, $\sigma_N$ are arbitrary,
their norms in the upper bound can be replaced with
the corresponding trace norms of the actual boundary data.
\end{theorem}

\begin{proof}
The Dirichlet boundary condition can be dealt with by using extension
$u_D\in H^1(\Omega;\Rt)$ to add $\dual{\trnn(u_D)}{\dsigma}_\cS$ on the right-hand side of
\eqref{VFa}, and replacing $\eta$ with $\wtilde\eta:=\eta-\trnn(u_D)\in H^{1/2}_{D,t}(\cS;\Rt)$.

We verify the conditions for mixed formulations.

{\bf Uniform boundedness.}
The uniform boundedness of the right-hand side functionals
$\dual{\trnn(u_D)}{\cdot}_\cS$, $\vdual{f}{\cdot}$, $\dual{\cdot}{\sigma_N}_\cS$
and the bilinear forms on the left-hand sides of \eqref{VFa} and \eqref{VFb}
follows from the underlying norms associated with the product and trace spaces.

{\bf Inf-sup condition.}
The bound
\begin{align} \label{infsup_div}
   \sup_{0\not=\tau\in H_N(\div,\Omega;\SS)}
   \frac {\vdual{\div\tau}{\du}}{\|\tau\|_{\div}}
   \gtrsim
   \|\du\|\quad\forall\du\in L_2(\Omega;\Rt)
\end{align}
holds by Korn's inequality \eqref{Korn}, and
\begin{align} \label{infsup_trace}
   \sup_{0\not=\tau\in H_{nn}(\div,\mesh;\SS)}
   \frac {\dual{\deta}{\tau}_\cS}{\|\tau\|_{\div,\mesh}}
   \ge
   \|\deta\|_{1/2,t,\cS}\quad\forall\deta\in H^{1/2}_{D,t}(\cS;\Rt)
\end{align}
holds due to Proposition~\ref{prop_norm}.
An application of \cite[Theorem~3.3]{CarstensenDG_16_BSF}
(with $Y:=H_{nn}(\div,\mesh;\SS)$, $Y_0:=H_N(\div,\Omega;\SS)$, $X_0:=L_2(\Omega;\Rt)$,
and $\widehat X:= H^{1/2}_{D,t}(\cS;\Rt)$)
shows that a combination of estimates \eqref{infsup_div}, \eqref{infsup_trace}
with Lemma~\ref{la_cont} implies the required inf-sup condition
\begin{align} \label{infsup}
   \sup_{0\not=\tau\in H_{nn}(\div,\mesh;\SS)}
   \frac {\vdual{\div\tau}{\du}_\mesh -\dual{\deta}{\tau}_\cS}{\|\tau\|_{\div,\mesh}}
   \gtrsim
   \|\du\| + \|\deta\|_{1/2,t,\cS}\quad
   \forall\du\in L_2(\Omega;\Rt),\ \deta\in H^{1/2}_{D,t}(\cS;\Rt).
\end{align}
In the case $\G{ss}\cup\G{sf}=\emptyset$, we will need this inf-sup property for the
constrained space
\begin{align} \label{Hconstrained}
   \wtilde H_{nn}(\div,\mesh;\SS) := \{
   \tau\in H_{nn}(\div,\mesh;\SS);\; \vdual{\tr(\tau)}{1}=0\}.
\end{align}
By Proposition~\ref{prop_norm}, the constrained form of inf-sup property
\eqref{infsup_trace} is valid in this case as well.
To verify the constrained version of \eqref{infsup_div}, let
$0\not=\tau\in H_N(\div,\Omega;\SS)$ be given. Since $\G{ss}\cup\G{sf}=\emptyset$,
the only boundary condition such a function $\tau$ must satisfy is $\bt\cdot\tau\bn=0$ on $\G{sc}$.
We introduce the trace projection $\tau_0$ of $\tau$ by
\[
   \tau_0:=\frac 1{2|\Omega|} \vdual{\tr(\tau)}{1}\id
\]
such that $\vdual{\tr(\tau-\tau_0)}{1}=0$.
By the orthogonality of $\bt$ and $\bn$, $\bt\cdot\tau_0\bn=0$ on $\G{sc}$.
It follows that $\wtilde\tau:=\tau-\tau_0\in H_N(\div,\Omega;\SS)$
and $\vdual{\tr(\wtilde\tau)}{1}=0$.
In particular, $\wtilde\tau\in \wtilde H_{nn}(\div,\mesh;\SS)$ and
$\|\tau\|^2=\|\wtilde\tau\|^2+\|\tau_0\|^2$.
In this way we obtain the constrained version of inf-sup property \eqref{infsup_div},
\[
   \sup_{0\not=\tau\in H_N(\div,\Omega;\SS)}
   \frac {\vdual{\div(\tau-\tau_0)}{\du}}{\|\tau-\tau_0\|_{\div}}
   \ge
   \sup_{0\not=\tau\in H_N(\div,\Omega;\SS)}
   \frac {\vdual{\div\tau}{\du}}{\|\tau\|_{\div}}
   \gtrsim
   \|\du\|\quad\forall\du\in L_2(\Omega;\Rt).
\]
We then proceed as before and use \cite[Theorem~3.3]{CarstensenDG_16_BSF}
to deduce the constrained version of inf-sup property \eqref{infsup}
where $H_{nn}(\div,\mesh;\SS)$ is replaced with $\wtilde H_{nn}(\div,\mesh;\SS)$.

{\bf Coercivity.}
For the well-posedness of \eqref{VF} it remains to verify the coercivity of
bilinear form $\vdual{\AA\cdot}{\cdot}$ on the kernel
\[
   K_0 :=
   \{\tau\in H_{nn}(\div,\mesh;\SS);\;
     \vdual{\div\tau}{\du}_\mesh=0\ \forall\du\in L_2(\Omega;\Rt),\;
     \dual{\deta}{\tau}_\cS=0\ \forall\deta\in H^{1/2}_{D,t}(\cS;\Rt)\}.
\]
It is clear that any $\tau\in K_0$ satisfies $\div_\mesh\tau=0$. By Lemma~\ref{la_cont}
we therefore conclude that
\begin{align*}
   K_0 = \{\tau\in H_N(\div,\Omega;\SS);\; \div\tau=0\}.
\end{align*}
Relation \cite[(9.1.8)]{BoffiBF_13_MFE},
\begin{align} \label{coercive}
   \vdual{\AA\tau}{\dtau}
   =
   \frac 1{2\mu} \vdual{\dev\tau}{\dev\dtau}
   +
   \frac 1{4(\lambda+\mu)} \vdual{\tr(\tau)}{\tr(\dtau)}
   \quad\forall \tau,\dtau\in L_2(\Omega;\SS),
\end{align}
shows that bilinear form $\vdual{\AA\cdot}{\cdot}$
is coercive on $K_0$, though the ellipticity constant may depend on $\lambda$.
It follows that \eqref{VF} has a unique solution
$(\sigma,u,\eta)$.

To show the {\bf uniform stability} (independent of $\nu$) of \eqref{VF}
we use \eqref{coercive} to deduce the well-known estimate
\begin{align} \label{coercive_A}
   \vdual{\AA\tau}{\tau}\ge \frac 1{2\mu} \|\dev\tau\|^2
   \quad\forall \tau\in L_2(\Omega;\SS),
\end{align}
and are left with controlling $\|\tr(\tau)\|$. To this end we distinguish between two cases.

{\bf 1. Case $\G{ss}\cup\G{sf}\not=\emptyset$.}
It is clear that $K_0$ is a closed subspace of $H(\div,\Omega;\Rtt)$.
Recall the general assumption that any part of the splitting \eqref{split} is either empty
or has a positive (surface) measure. Therefore, in the current case at least one of the
boundary components $\G{ss}$ or $\G{sf}$ has positive measure. Thus, a function
$v\in H^1_D(\Omega;\Rt)$ exists with $\dual{v\cdot\bn}{1}_{\G{ss}\cup\G{sf}}\not=0$.
Then,
\[
   \dual{v\cdot\bn}{\bn\cdot\id\bn}_\G{ss} + \dual{v}{\id\bn}_\G{sf}
   = \dual{v\cdot\bn}{1}_{\G{ss}\cup\G{sf}}\not=0
\]
proves that $\id\not\in K_0$.
Hence, Lemma~\ref{la_trdevdiv} is applicable with $s=0$, $\Sigma:=K_0$,
and relation \eqref{coercive_A} gives
\begin{align*}
   \vdual{\AA\tau}{\tau}
   \gtrsim \|\dev\tau\|^2 + \|\tr(\tau)\|^2
   \simeq \|\tau\|_{\div,\mesh}^2 \quad\forall\tau\in K_0,
\end{align*}
uniformly with respect to $\nu$ and $\mesh$.
We conclude that the solution of \eqref{VF} is uniformly bounded.

{\bf 2. Case $\G{ss}\cup\G{sf}=\emptyset$.} 
The selection of $\dsigma=\id$ in \eqref{VFa} shows that
\begin{align*}
   &0=\vdual{\AA\sigma}{\id} + \vdual{u}{\div\id}_\mesh - \dual{\eta}{\id}_\cS
   =
   \frac 1{2(\lambda+\mu)} \vdual{\tr(\sigma)}{1}
   - \dual{\g{D}\cdot\bn}{1}_\G{hc} - \dual{\g{D,n}}{1}_\G{sc}
\end{align*}
because
\begin{align} \label{dualid}
   \dual{\eta}{\id}_\cS = \dual{\eta}{\bn}_\Gamma
   = \dual{\eta}{\bn}_\G{hc} + \dual{\eta\cdot\bn}{1}_\G{sc}
   = \dual{\g{D}}{\bn}_\G{hc} + \dual{\g{D,n}}{1}_\G{sc}
\end{align}
by \eqref{VFbc}. Therefore,
\(
   \vdual{\tr(\sigma)}{1}
   =
   2(\lambda+\mu)\Bigl(\dual{\g{D}\cdot\bn}{1}_\G{hc} + \dual{\g{D,n}}{1}_\G{sc}\Bigr).
\)
We conclude that tensor $\sigma_0$ defined in \eqref{sigma0}
is a datum that provides the trace projection of $\sigma$,
\begin{align} \label{trsigma0}
   \sigma_0
   = 
   \frac {\lambda+\mu}{|\Omega|}
   \Bigl(\dual{\g{D}\cdot\bn}{1}_\G{hc} + \dual{\g{D,n}}{1}_\G{sc}\Bigr)\id
   =\frac 1{2|\Omega|}\vdual{\tr(\sigma)}{1}\id,
\end{align}
that is, $\vdual{\tr(\sigma_0)}{1}=\vdual{\tr(\sigma)}{1}$.
In what follows, we take into account that $\sigma_0$ may not vanish
but, if $\sigma_0=0\in\SS$, some of the calculations are trivial.

We rewrite system \eqref{VF} by replacing $\sigma$ with
$\wtilde\sigma:=\sigma-\sigma_0\in\wtilde H_{nn}(\div,\mesh;\SS)$ (cf.~\eqref{Hconstrained}),
giving the system
\begin{subequations} \label{VF2}
\begin{alignat}{3}
   &\vdual{\AA\wtilde\sigma}{\dsigma} + \vdual{u}{\div\dsigma}_\mesh
   - \dual{\wtilde\eta}{\dsigma}_\cS
   &&= \dual{\trnn(u_D)}{\dsigma}_\cS,\label{VFa2}\\
   &\vdual{\div\wtilde\sigma}{\du}_\mesh - \dual{\deta}{\wtilde\sigma}_\cS
   &&= -\vdual{f}{\du}-\dual{\deta}{\sigma_N}_\cS\label{VFb2}
\end{alignat}
\end{subequations}
for $\du\in L_2(\Omega;\Rt)$, $\deta\in H^{1/2}_{D,t}(\cS;\Rt)$, and
$\dsigma\in \wtilde H_{nn}(\div,\mesh;\SS)$.
Here, for $\sigma_0\not=0$, we used that
$\vdual{\sigma_0}{\dsigma}=0=\dual{\deta}{\sigma_0}_\cS$ for such test
functions. The first relation is due to the new constraint for $\dsigma$
and the second is due to the homogeneous Dirichlet traces of $\deta$, cf.~\eqref{dualid}.

Formulation \eqref{VF2} is equivalent to system \eqref{VF} with solution
$(\wtilde\sigma,u,\wtilde\eta)=(\sigma-\sigma_0,u,\eta-\trnn(u_D))$.
The kernel space becomes
\[
   \wtilde K_0 := \{\tau\in K_0;\; \vdual{\tr(\tau)}{1}=0\},
\]
and Lemma~\ref{la_trdevdiv} with $s=0$ and $\Sigma:=\wtilde K_0$ shows that
\[
   \|\tr(\tau)\| \le
   \ctdd\bigl( \|\dev\tau\| + \|\div\tau\|_{-1}\bigr)
   = \ctdd \|\dev\tau\|\quad\forall \tau\in \wtilde K_0.
\]
We then continue with \eqref{coercive_A} to bound
\begin{align*}
   \vdual{\AA\tau}{\tau}\gtrsim
   \|\dev\tau\|^2 + \|\tr(\tau)\|^2
   \simeq
   \|\tau\|_{\div,\mesh}^2\quad\forall\tau\in \wtilde K_0.
\end{align*}
This is the uniform coercivity of $\vdual{\AA\cdot}{\cdot}$ on $\wtilde K_0$,
and we conclude that the solution $(\wtilde\sigma,u,\wtilde\eta)$ of \eqref{VF2}
is uniformly bounded as
\begin{align} \label{b1}
   \|\wtilde\sigma\|_{\div,\mesh} + \|u\| + \|\wtilde\eta\|_{1/2,t,\cS}
   \lesssim \|f\| + \|u_D\|_1 + \|\sigma_N\|_{\div}.
\end{align}
The bound $\|\eta\|_{1/2,t,\cS}\le\|\wtilde\eta\|_{1/2,t,\cS}+\|u_D\|_1$
leads to the claimed estimate in the case of $\G{ss}\cup\G{sf}=\emptyset$.

We finish the proof of the theorem by noting that
the regularities $u\in H^1(\Omega;\Rt)$ and $\sigma\in H(\div,\Omega;\SS)$ are immediate,
as is the fact that they constitute a (weak) solution of \eqref{model}.
Relation $\eta=\trnn(u)$ then follows from \eqref{VFa} by definition of $\trnn$.
\end{proof}

\section{An $H_{nn}(\div,\mesh;\SS)$ conforming approximation space} \label{sec_approx}

\subsection{Polynomial spaces and transformations} \label{sec_pol}

We use notation $P^k(\omega;U)$ for the space of $U$-valued functions whose components
are polynomials of total degree at most $k$.
Here, $\omega\subset\Rt$ is a triangle or an interval (line segment), and
$U\in\{\R,\Rt,\SS\}$. Occasionally we drop the argument $U$ when $U=\R$, and will also drop $\omega$
when the domain of definition is clear.
We generically denote by $\Pi^k_\el$ the $L_2$ projection
onto $P^k(\el;U)$ ($k\in\{0,1\}$, $U\in\{\Rt,\SS\}$),
and use the corresponding piecewise projection $\Pi^k_\mesh$.
For integers $i,j\ge 0$, $x^iy^j$ is understood as the polynomial
$(x,y)\mapsto x^iy^j$ with generic domain as needed.
The reference triangle is $\elref$ with vertices
$\widehat x_1=(0,1)$, $\widehat x_2=(0,0)$, $\widehat x_3=(1,0)$ and
edges $\widehat e_j$ which are opposite to $\widehat x_j$ ($j=1,2,3$).
As introduced earlier, $\bn$ and $\bt$ is our generic notation for
normal and tangential vectors along edges of the mesh.

We need appropriate transformations of tensor and vector functions.
To distinguish between objects, functions, and differential
operators acting on or associated with $\elref$ or a triangle $\el\in\mesh$, we add the symbol
 ``\,$\widehat{\ }$\;'' to $\elref$-related objects. For instance,
$\strainref{\widehat v}$ is the strain tensor or symmetric gradient of
$\widehat v\in H^1(\elref;\Rt)$.
For $\el\in\mesh$ let $F_\el:\;\elref\to T$ denote the affine diffeomorphism
\begin{align*}
  \begin{pmatrix}\widehat x\\ \widehat y\end{pmatrix}
  \mapsto
  B_\el \begin{pmatrix}\widehat x\\ \widehat y\end{pmatrix} + a_\el,
\end{align*}
where $a_\el\in\Rt$, $B_\el\in\Rtt$, and set $J_\el:=\det(B_\el)\not=0$.
In the following, we only consider transformations which generate families of shape-regular
elements, $h_\el$ refers to the diameter of $\el$, and $h_\mesh\in L_\infty(\Omega)$ with
$h_\mesh|_\el:=h_\el$ ($\el\in\mesh$) is the mesh-width function.
We also assume that $|J_\el|$ is uniformly bounded so that the
$L_\infty(\Omega)$ norm $\|h_\mesh\|_{\infty}=O(1)$ uniformly for all meshes $\mesh$.

We transform tensors $\widehat\tau:\;\elref\to\SS$ onto $\el\in\mesh$
by the (re-scaled) contravariant Piola--Kirchhoff transformation,
\begin{align} \label{PiolaK}
  \PiolaK{\el}(\widehat\tau)=\tau\quad\text{with}\quad
  J_\el^2 \tau\circ F_\el := B_\el\widehat\tau B_\el^\T,
\end{align}
and also use the covariant form
\begin{align*}
  \PiolaKc{\el}(\widehat\tau)=\tau\quad\text{with}\quad
  \tau\circ F_\el := J_\el B_\el^{-\T}\widehat\tau B_\el^{-1}.
\end{align*}
Vector functions $\widehat v:\;\elref\to\Rt$ are transformed by the (re-scaled, covariant)
Piola transformation
\begin{align} \label{Piolacurl}
  &\PiolaCurl{\el}(\widehat v)=v\quad\text{with}\quad
  v\circ F_\el := J_\el B_\el^{-\T}\widehat v. 
\end{align}
Note that, as selected by Pechstein and Sch\"oberl \cite{PechsteinS_11_TDN}, there is a factor $J_\el^2$
rather than $J_\el$ in \eqref{PiolaK} as in \cite{FuehrerH_19_FDD}. This is due to the fact
that we are aiming at the pointwise continuity of the normal-normal traces of tensors rather than
the continuity in the sense of traces that are dual to traces of $H^2$ functions,
intrinsic to the Kirchhoff--Love plate bending model, or the bi-Laplacian.
There is a corresponding non-standard factor $J_\el$ in Piola transformation
\eqref{Piolacurl}. It makes it unsuitable for a curl-conforming mapping, not needed here.
There does not seem to be agreement on the use of the terms co- and contravariant,
cf.~\cite{PechsteinS_11_TDN,RognesKL_09_EAH}.

For convenient reading we collect some transformation and conservation properties
in the next lemma.

\begin{lemma}\label{la_trafo}
Let $\el\in\mesh$. For $\widehat\tau\in H(\div,\elref;\SS)$ and
$\widehat v\in H^1(\elref;\Rt)$ set
$\tau:=\PiolaK{\el}(\widehat\tau)$ and $v:=\PiolaCurl{\el}(\widehat v)$.
Relations
\begin{align} \label{trafo_div_strain}
  &J_\el^2 \div\tau \circ F_\el = B_\el \divref\widehat\tau,\quad
  \strain{v}\circ F_\el =  J_\el B_\el^{-\T} \strainref{\widehat v}B_\el^{-1}
\end{align}
and
\begin{align*}
   \vdual{\div\tau}{v}_\el = \vdual{\divref\widehat\tau}{\widehat v}_\elref,\quad
   \vdual{\tau}{\strain{v}}_\el = \vdual{\widehat\tau}{\strainref{\widehat v}}_\elref,\quad
   \dual{\tau\bn}{v}_{\partial\el} = \dual{\widehat\tau\widehat\bn}{\widehat v}_{\partial\elref}
\end{align*}
hold. Furthermore, equivalences
\begin{alignat*}{3}
  & \|\tau\|_\el \simeq h_\el^{-1} \|\widehat\tau\|_\elref,
  \quad &&
  \|\div\tau\|_\el \simeq h_\el^{-2} \|\divref\widehat\tau\|_\elref
  \quad &&
  \forall \widehat\tau\in H(\div,\elref;\SS),\\
  &\quad &&
  |\tau|_{s,\el} \simeq h_\el^{-1-s} |\widehat\tau|_{s,\elref}
  \quad &&
  \forall \widehat\tau\in H^s(\elref;\SS),\ 0<s\le 1,\\
  & \|v\|_\el \simeq h_\el^2 \|\widehat v\|_\elref,
  \quad &&
  \|\strain{v}\|_\el \simeq h_\el\|\strainref{\widehat v}\|_\elref
  \quad &&
  \forall \widehat v\in H^1(\elref;\Rt)
\end{alignat*}
hold with equivalence constants independent of $\widehat\tau$, $\widehat v$, and $\el\in\mesh$.
\end{lemma}

\begin{proof}
For $\el\in\mesh$ let $\widehat\tau\in H(\div,\elref;\SS)$, $\widehat v\in H^1(\elref;\Rt)$ and
$\widehat\varphi\in C^\infty_0(\elref;\Rt)$ be given with corresponding transformations
$\tau$, $v$ and $\varphi=\PiolaCurl{\el}(\widehat\varphi)$ onto $\el$.
The strain relationship in \eqref{trafo_div_strain} follows from vector calculus,
cf.~\cite[Section~3.1]{PechsteinS_11_TDN}.
Therefore, integration by parts shows that
\[
   \vdual{\div\tau}{\varphi}_\el = -\vdual{\tau}{\strain{\varphi}}_\el
   = - \vdual{B_\el\widehat\tau B_\el^\T}{B_\el^{-\T}\strainref{\widehat\varphi}B_\el^{-1}}_\elref
   = - \vdual{\widehat\tau}{\strainref{\widehat\varphi}}_\elref
   =   \vdual{\divref\widehat\tau}{\widehat\varphi}_\elref.
\]
That is, $\div\tau\circ F_\el=J_\el^{-2}B_\el\divref\widehat\tau$
since $\widehat\varphi=J_\el^{-1}B_\el^\T\varphi\circ F_\el$, cf.~\cite[Lemma~9]{FuehrerH_19_FDD}.
We have just seen that
\[
   \vdual{\div\tau}{v}_\el = \vdual{\divref\widehat\tau}{\widehat v}_\elref
   \quad\text{and}\quad
   \vdual{\tau}{\strain{v}}_\el = \vdual{\widehat\tau}{\strainref{\widehat v}}_\elref,
\]
implying
\[
     \dual{\tau\bn}{v}_{\partial\el}
   = \vdual{\div\tau}{v}_\el + \vdual{\tau}{\strain{v}}_\el
   = \vdual{\divref\widehat\tau}{\widehat v}_\elref
   + \vdual{\widehat\tau}{\strainref{\widehat v}}_\elref
   = \dual{\widehat\tau\widehat\bn}{\widehat v}_{\partial\elref}.
\]
The scaling properties follow by routine arguments.
\end{proof}

\subsection{An $H(\div;\SS)$ element} \label{sec_el}

For a triangle $\el\in\mesh$, we consider the following local subspace of $P^2(\el;\SS)$,
\begin{equation} \label{P21}
   \XX{\el}:=\{\tau\in P^2(\el;\SS);\; \bn\cdot\tau\bn|_e\in P^1(e)\ \forall e\in \cE(\el)\},
\end{equation}
and assign the degrees of freedom
\begin{subequations} \label{dof_el}
\begin{alignat}{3}
   \label{dof_el_nn}
   &|e| \dual{\bn\cdot\tau\bn}{q}_{e}\quad &&(q\in P^1(e),\ e\in\cE(\el)),\\
   \label{dof_el_id}
   &\vdual{\tau}{c}_{\el}        &&(c\in P^0(\el;\SS)),\\
   \label{dof_el_div}
   &\vdual{\div\tau}{v}_{\el}    &&(v\in P^1(\el;\Rt)).
\end{alignat}
\end{subequations}

\begin{remark} \label{rem_3d}
The construction of space \eqref{P21} with degrees of freedom \eqref{dof_el}
can be extended to tetrahedra. In that case, the dimension of the
space is larger than the number of degrees of freedom that correspond to \eqref{dof_el}.
The three-dimensional case thus requires a slightly different space definition.
Its analysis is more technical and left to future research.
\footnote{Note after publication: 3d case treated in \url{http://arxiv.org/abs/2502.00609}.}
\end{remark}

It is important to know how degrees of freedom \eqref{dof_el} transform when mapping tensors between
triangles. This is  established by the following lemma.

\begin{lemma} \label{la_dof}
Given $\el\in\mesh$ and a sufficiently smooth tensor $\tau\in L_2(\el;\SS)$,
let $\widehat\tau\in L_2(\elref;\SS)$ be defined by $\tau=\PiolaK{\el}\widehat\tau$.
Furthermore, let $e\in\cE(\el)$, $q\in P^2(e)$, $c\in P^0(\el;\SS)$ and
$v\in P^1(\el;\Rt)$ be given and define
$\widehat e\in\cE(\elref)$, $\widehat q\in P^2(\widehat e)$, $\widehat c\in P^0(\elref;\SS)$ and
$\widehat v\in P^1(\elref;\Rt)$ such that $e=F_\el\widehat e$, $q\circ F_\el=\widehat q$,
$c=\PiolaKc{\el}\widehat c$ and $v=\PiolaCurl{\el}\widehat v$.
Then,
\begin{align*}
   &|e|\dual{\bn\cdot\tau\bn}{q}_{e}
   = |\widehat e|\dual{\widehat\bn\cdot\widehat\tau\widehat\bn}{\widehat q}_{\widehat e},\quad
   \vdual{\tau}{c}_{\el}=\vdual{\widehat\tau}{\widehat c}_\elref,\quad
   \vdual{\div\tau}{v}_{\el}=\vdual{\divref\widehat\tau}{\widehat v}_\elref.
\end{align*}
(Recall that $\bn$, resp. $\widehat\bn$, is the normal vector on $e$, resp. $\widehat e$.)
\end{lemma}

\begin{proof}
We have seen the transformation of $\vdual{\div\tau}{v}_\el$ in Lemma~\ref{la_trafo}.
The other relations follow by replacing $\tau=\PiolaK{\el}\widehat\tau$ and the test functions
as defined, and noting that
\(
   \bn=\frac{J_\el}{J_e}B_\el^{-\T}\widehat\bn.
\)
Here, $J_e$ denotes the signed ratio of the lengths of edges $e$ divided by $\widehat e$.
\end{proof}

\begin{lemma} \label{la_dof_el}
Degrees of freedom \eqref{dof_el} are unisolvent for $\XX{\el}$, $\el\in\mesh$.
\end{lemma}

\begin{proof}
To prepare the proof let us introduce some notation and tensors.
Given $\el\in\mesh$, we denote its edges, normal and tangential vectors by
$e_i$, $\bn_i$ and $\bt_i$ ($i=1,2,3$), respectively. The barycentric coordinates
are $\lambda_i$, numbered so that they are assigned to vertices opposite to corresponding edges.
All objects are numbered in positive orientation modulo $3$.
We introduce tensors
\[
   T_i:=\frac{\bt_{i+1}\odot\bt_{i+2}}{\bn_i\cdot\bt_{i+1} \bn_i\cdot\bt_{i+2}}\quad (i=1,2,3),
\]
where
$\bt_{j}\odot\bt_{k}:=\mathrm{sym}(\bt_j\otimes\bt_k):=
(\bt_{j}\bt_{k}^\T+\bt_{k}\bt_{j}^\T)/2$ is the symmetrized dyadic product
of vectors $\bt_j$ and $\bt_k$. It is immediate that $\bn_j\cdot T_i\bn_j|_{e_j}=\delta_{ij}$
for $i,j=1,2,3$. Therefore, $T_1,T_2,T_3$ form a basis of $P^0(\el;\SS)$.
We introduce the following tensors,
\begin{alignat*}{3}
   &N_{i,1}:= T_i,\quad
   &&N_{i,2}:= (\lambda_{i+2}-\lambda_{i+1})T_i,\quad
   &&N_{i,3}:= (\lambda_{i+2}-\lambda_{i+1})^2T_i,\\
   &B_{i,1}:=\lambda_i T_i,\quad
   &&B_{i,2}:=\lambda_i\lambda_{i+1} T_i,\quad
   &&B_{i,3}:=\lambda_i\lambda_{i+2} T_i
\end{alignat*}
for $i\in\{1,2,3\}$.
We observe that $N_{i,1},N_{i,2},N_{i,3}$ have zero normal-normal traces on edges $e_j$ ($j\not=i$),
and have non-zero polynomial normal-normal trace on $e_i$ of degrees $0$, $1$, and $2$,
respectively. In particular, $(N_{i,j}; i,j\in\{1,2,3\})$ are linearly independent.
Since $\lambda_iT_i$ has vanishing normal-normal traces on all the edges of $\el$
($i\in\{1,2,3\}$), and using the linear independence of the three functions
$\lambda_i,\lambda_i\lambda_{i+1},\lambda_i\lambda_{i+2}$ (for fixed $i\in\{1,2,3\}$), we find that
$(B_{i,j}; i,j\in\{1,2,3\})$ form a set of $9$ linearly independent tensors with
vanishing normal-normal traces along $\partial\el$.
Therefore, $\{N_{i,j}, B_{i,j};\; i,j=1,2,3\}$ is a basis of $P^2(\el;\SS)$,
and the subset $\{N_{i,k}, B_{i,j};\; i,j=1,2,3,\ k=1,2\}$ is a basis of $\XX{\el}$.

The dimension of $\XX{\el}$ is $15$ and equals the number of degrees of freedom \eqref{dof_el}:
$6$ in \eqref{dof_el_nn}, $3$ in \eqref{dof_el_id}, and $6$ in \eqref{dof_el_div}.
To prove the lemma it is therefore enough to show injectivity.
To this end, let $\tau\in\XX{\el}$ be given with vanishing degrees of freedom \eqref{dof_el}.

The vanishing of degrees \eqref{dof_el_div} means that $\div\tau=0$. Therefore,
$\tau=\Curl\curl\beta$ with $\beta\in H^2(\el)$, see~\cite[Lemma~3.2]{CarstensenD_98_PEE}.
Here, $\curl\beta:=(\grad\beta)^\perp$ (as defined previously) and
$\Curl\begin{pmatrix}v_1\\v_2\end{pmatrix}
:=\begin{pmatrix}\curl v_1^\T\\ \curl v_2^\T\end{pmatrix}$.
Since $\tau\in P^2(\el;\SS)$ we have $\beta\in P^4(\el)$.
Furthermore, $\beta$ is unique up to linear polynomials so that we can assume
that $\beta$ vanishes at the vertices of $\el$.
Since the normal-normal traces of $\tau$ on the edges vanish it follows that
the second-order tangential derivatives of $\beta$ on the edges vanish as well.
That is, $\beta|_e\in P^1(e)$ for $e\in\cE(\el)$, thus $\beta|_{\partial\el}=0$ since it
vanishes at the vertices. We have shown that $\beta=b_\el q$ holds for the element
bubble function $b_\el:=\lambda_1\lambda_2\lambda_3$ and some polynomial $q\in P^1(\el)$.

To conclude, we note that $\beta|_{\partial\el}=0$ gives
$\grad\beta|_e = (\grad\beta\cdot\bn)\bn|_e$, thus
$\curl\beta|_e=(\grad\beta)^\perp|_e=-(\grad\beta\cdot\bn)\bt|_e$ for $e\in\cE(\el)$.
Therefore, the vanishing of degrees of freedom \eqref{dof_el_id} implies
\begin{align*}
   0 &= \int_\el \tau\, \mathrm{d}(x,y) = \int_\el \Curl\curl\beta\, \mathrm{d}(x,y)
   = -\int_{\partial\el} \curl\beta\otimes\bt\,\mathrm{d}s
   = \sum_{e\in\cE(\el)} b_e \bt_e\otimes\bt_e
\end{align*}
with $b_e:=\int_e \grad\beta\cdot\bn\,\mathrm{d}s$, $e\in\cE(\el)$.
Since $\{\bt_e\otimes\bt_e;\; e\in \cE(\el)\}$ is a basis of $P^0(\el;\SS)$ we conclude
that $b_e=0$ for all $e\in\cE(\el)$. Recall that $\beta=b_\el q$ with element bubble
$b_\el$ and linear polynomial $q$. One finds that
$(\grad b_\el\cdot\bn)|_e=-d_e b_e$ with edge bubble $b_e\in P^2(e)$ and distance
$d_e$ of edge $e$ to the opposite vertex. This means that
\[
   0 = \int_e \grad\beta\cdot\bn\,\mathrm{d}s
     = -d_e \int_e q b_e\,\mathrm{d}s\quad\forall e\in\cE(\el).
\]
It follows that $q$ vanishes at the midpoints of the three edges. Therefore, $q=0$
and $\beta=q b_\el=0$, thus $\tau=0$. This finishes the proof.
\end{proof}

Degrees of freedom \eqref{dof_el} induce a local interpolation operator
\[
   \Inn{\el}:\; H(\div,\el;\SS)\cap L_q(\el;\SS) \to \XX{\el} \quad (\el\in\mesh,\ q>2).
\]
It has the following properties.

\begin{lemma} \label{la_Inn_el}
Given $q>2$ and $\el\in\mesh$,
operator $\Inn{\el}:\;H(\div,\el;\SS)\cap L_q(\el;\SS)\to \XX{\el}$ is
well defined and bounded. It commutes with the divergence operator,
\(
   \div\Inn{\el}=\Pi^1_\el\div,
\)
and has the projection property
$\Inn{\el}\tau=\tau$ $\forall\tau\in \XX{\el}$.
Furthermore, if $\tau\in H^{s,r}(\div,\el;\SS)$ with $0<s\le 1$ and $0\le r\le 2$,
the error estimates
\begin{alignat*}{2}
   \|\tau-\Inn{\el}\tau\|_{\el}
   &\lesssim h_\el^s \|\tau\|_{s,\el} + h_\el \|\div\tau\|_\el\quad
   &&\forall \tau\in H^{s,0}(\div,\el;\SS),\\
   \|\div(\tau-\Inn{\el}\tau)\|_\el
   &\lesssim h_\el^r \|\div\tau\|_{r,\el}
   \quad&&\forall \tau\in H^{s,r}(\div,\el;\SS)
\end{alignat*}
hold true uniformly for $\el\in\mesh$ and $\tau$ with the given regularity.
\end{lemma}

\begin{proof}
The interpolation operator is well defined and bounded
for $\tau\in H(\div,\el;\SS)\cap L_q(\el;\SS)$ if $q>2$.
Indeed, the boundedness of degrees of freedom \eqref{dof_el_id} and \eqref{dof_el_div}
is immediate, and the moments of normal traces $\tau\bn$
on edges are also bounded, cf.~\cite[Lemma~4.7]{AmroucheBDG_98_VPT}.
The commutativity property holds due to degrees of freedom
\eqref{dof_el_div}.
The projection property is due to the unisolvency of the degrees of freedom.

Note that $H^s(\el;\SS)\subset L_q(\el;\SS)$ for
$q=2/(1-s)$ if $s<1$ and any $q>2$ if $s=1$ by Sobolev's embedding theorem,
see~\cite[Theorem~1.4.5.2]{Grisvard_85_EPN}.
Therefore, given $0<s\le 1$, $0\le r\le 2$, $\el\in\mesh$ and $\tau\in H^{s,r}(\div,\el;\SS)$,
$\Inn{\el}\tau$ is well defined.
We start with the triangle inequality and a standard error estimate to see that
\begin{align*}
   \|\tau-\Inn{\el}\tau\|_\el
   \le \|\tau - \Pi^0_\el\tau\|_\el + \|\Inn{\el}(\tau-\Pi^0_\el\tau)\|_\el
   \lesssim h_\el^s |\tau|_{s,\el} + \|\Inn{\el}(\tau-\Pi^0_\el\tau)\|_\el.
\end{align*}
Degrees of freedom \eqref{dof_el} are invariant
under appropriate transformations as specified in Lemma~\ref{la_dof}.
In particular, we have the following commuting properties,
\[
   \Inn{\el}\PiolaK{\el}\widehat\tau=\PiolaK{\el}\Inn{\elref}\widehat\tau, \quad
   \Pi^0_\el\PiolaK{\el}\widehat\tau=\PiolaK{\el}\Pi^0_\elref\widehat\tau.
\]
Therefore, using the scaling of the $L_2$ norm of tensors given by Lemma~\ref{la_trafo},
the boundedness of the trace operator (mentioned previously) on the reference element,
and a Bramble--Hilbert argument, we conclude that
\begin{align*}
   h_\el^2 \|\Inn{\el}(\tau-\Pi^0_\el\tau)\|_\el^2
   &\simeq \|\Inn{\elref}(\widehat\tau-\Pi^0_\elref\widehat\tau)\|_\elref^2
   \lesssim
   \|\widehat\bn\cdot(\widehat\tau-\Pi^0_\elref\widehat\tau)\widehat\bn\|_{\partial\elref}^2
   +
   \|\divref\widehat\tau\|_\elref^2
   \lesssim
   \|\divref\widehat\tau\|_\elref^2 + |\widehat\tau|_{s,\elref}^2.
\end{align*}
Transforming back to $\el$ and using again scaling properties given by Lemma~\ref{la_trafo}, we obtain
\[
   h_\el^2 \|\Inn{\el}(\tau-\Pi^0_\el\tau)\|_\el^2
   \lesssim
   h_\el^4 \|\div\tau\|_\el^2 + h_\el^{2+2s} |\tau|_{s,\el}^2.
\]
This yields the stated bound for $\|\tau-\Inn{\el}\|_\el$.
The proof is finished by using the commutativity property to see that
\[
   \|\div(\tau-\Inn{\el}\tau)\|_\el = \|\div\tau-\Pi^1_\el\div\tau\|_\el
   \lesssim
   h_\el^r\|\div\tau\|_{r,\el}.
\]
\end{proof}

Finally, we note the following conservation property of $\Inn{\el}$ which will be relevant for 
approximation estimates.

\begin{lemma} \label{la_Inn_el_t}
Operator $\Inn{\el}$ conserves linear moments of normal traces in the following sense.
Any $\tau\in H(\div,\el;\SS)\cap L_q(\el;\SS)$ with $q>2$ and $\el\in\mesh$ satisfies
\[
   \dual{\Inn{\el}(\tau)\bn}{v}_{\partial\el}
   =
   \dual{\tau\bn}{v}_{\partial\el}
   \quad\forall v\in P^1(\el;\Rt).
\]
\end{lemma}

\begin{proof}
Degrees of freedom \eqref{dof_el} imply that any $v\in P^1(\el;\Rt)$ satisfies
\begin{align*}
   \dual{\bn\cdot\Inn{\el}(\tau)\bn}{v\cdot\bn}_{\partial\el}
   &=
   \dual{\bn\cdot\tau\bn}{v\cdot\bn}_{\partial\el} \quad\text{and}\\
   \dual{\bt\cdot\Inn{\el}(\tau)\bn}{v\cdot\bt}_{\partial\el}
   &=
   \dual{\Inn{\el}(\tau)\bn}{v}_{\partial\el}
   -\dual{\bn\cdot\Inn{\el}(\tau)\bn}{v\cdot\bn}_{\partial\el}\\
   &=
   \vdual{\div \Inn{\el}(\tau)}{v}_\el + \vdual{\Inn{\el}(\tau)}{\strain{v}}_\el
   -\dual{\bn\cdot\Inn{\el}(\tau)\bn}{v\cdot\bn}_{\partial\el}\\
   &=
   \vdual{\div \tau}{v}_\el + \vdual{\tau}{\strain{v}}_\el
   -\dual{\bn\cdot\tau\bn}{v\cdot\bn}_{\partial\el}
   =
   \dual{\bt\cdot\tau\bn}{v\cdot\bt}_{\partial\el}.
\end{align*}
Summation proves the statement.
\end{proof}

\subsection{The approximation space}

We use the local spaces $\XX{\el}$ from \eqref{P21} with corresponding product space $\XX{\mesh}$
 to define our conforming finite element space
\[
   P^2_{nn}(\mesh) := \XX{\mesh}\cap H_{nn}(\div,\mesh;\SS)
   = \{\tau\in\XX{\mesh};\; [\bn\cdot\tau\bn]|_e=0\ \forall e\in\cE(\Omega)\}.
\]
Here, $[\bn\cdot\tau\bn]|_e$ is the jump (with a certain orientation) of the normal-normal
trace of $\tau$ across edge $e$.
Of course, this conformity is achieved by assigning unique degrees of freedom \eqref{dof_el_nn}
on the interior edges of the mesh. This defines the global interpolation operator
\[
   \Inn{\mesh}:\; \begin{cases}
         H(\div,\Omega;\SS)\cap L_q(\Omega;\SS) &\to\quad P^2_{nn}(\mesh),\\
         \qquad\qquad \tau
         &\mapsto \bigl(\Inn{\mesh}\tau\bigr)|_\el:= \Inn{\el}\tau|_\el\ \text{for}\ \el\in\mesh
                 \end{cases}
   \quad (q>2).
\]
Direct application of Lemma~\ref{la_Inn_el} gives the following result.

\begin{prop} \label{prop_Inn}
Operator $\Inn{\mesh}:\;H(\div,\Omega;\SS)\cap L_q(\Omega;\SS)\to P^2_{nn}(\mesh)$ is
well defined and bounded for $q>2$.
It commutes with the piecewise divergence operator,
\(
   \div_\mesh\Inn{\mesh}=\Pi^1_\mesh\div,
\)
and has the projection property
$\Inn{\mesh}\tau=\tau$ for $\tau\in P^2_{nn}(\mesh)$.
Given $0<s\le 1$, $0\le r\le 2$, any $\tau\in H^{s,r}(\div,\Omega;\SS)$ satisfies
\begin{alignat*}{2}
   \|\tau-\Inn{\mesh}\tau\|
   &\lesssim \|h_\mesh\|_{\infty}^s \|\tau\|_s +  \|h_\mesh\|_{\infty} \|\div\tau\|\quad
   &&\forall \tau\in H^{s,0}(\div,\Omega;\SS),\\
   \|\div(\tau-\Inn{\mesh}\tau)\|_\mesh
   &\lesssim \|h_\mesh\|_\infty^r \|\div\tau\|_r
   \quad&&\forall \tau\in H^{s,r}(\div,\Omega;\SS)
\end{alignat*}
with intrinsic constants independent of $\tau$ and $\mesh$.
\end{prop}

\begin{remark} \label{rem_Inn}
By degrees of freedom \eqref{dof_el_id} it is clear that
$\Inn{\mesh}$ has the global projection property
$\vdual{\Inn{\mesh}\tau}{c}=\vdual{\tau}{c}$ for any
$\tau\in H(\div,\Omega;\SS)\cap L_q(\Omega;\SS)$ with $q>2$ and
any $c\in P^0(\Omega;\SS)$; in particular,
$\vdual{\tr(\Inn{\mesh}\tau)}{1}=\vdual{\tr(\tau)}{1}$.
\end{remark}

The local conservation property given by Lemma~\ref{la_Inn_el_t} transforms into
the following result.

\begin{lemma} \label{la_Inn_cons}
Any $\tau\in H(\div,\Omega;\SS)\cap L_q(\Omega;\SS)$ with $q>2$ satisfies
\[
   \dual{\trnn(v)}{\Inn{\mesh}(\tau)}_\cS
   = \dual{\trnn(v)}{\tau}_\cS = \dual{v}{\tau\bn}_\Gamma
   \quad\forall v\in P^1(\mesh;\Rt)\cap H^1(\Omega;\Rt).
\]
\end{lemma}

\begin{proof}
The first identity follows from Lemma~\ref{la_Inn_el_t} by summation over $\el\in\mesh$,
and the second one is simply due to the regularity of $\tau$ and Green's identity.
\end{proof}

The next lemma is critical to establish robustness of our mixed finite element
scheme. It is a discrete version of the trace-dev-div Lemma~\ref{la_trdevdiv}
for order $s=0$.

\begin{lemma} \label{la_trdevdiv2}
Let $\Sigma\subset L_2(\Omega;\SS)$ be a closed subspace that does not contain $\id$.
Furthermore, let $\mesh$ be any regular, shape-regular triangular mesh that covers $\Omega$.
Bound
\[
   \|\tr(\tau)\| \lesssim \|\dev\tau\|
\]
holds true uniformly for any $\tau\in\Sigma\cap P^2_{nn}(\mesh)$ with
\begin{align} \label{trdevdiv_ass}
   \div_\mesh\tau=0
   \quad\text{and}\quad
   \dual{\trnn(w)}{\tau}_\cS=0\quad\forall w\in P^1(\mesh;\R^2)\cap H^1_0(\Omega;\R^2).
\end{align}
The intrinsic constant appearing in the bound is independent of $\mesh$.
\end{lemma}

\begin{proof}
By Lemma~\ref{la_trdevdiv}
it suffices to bound $\|\div\tau\|_{-1}\lesssim \|\dev\tau\|$
for $\tau\in\Sigma\cap P^2_{nn}(\mesh)$ that satisfies \eqref{trdevdiv_ass}.
To this end, consider any $v\in H^1_0(\Omega;\Rt)$ and its Scott--Zhang interpolant
$v_h\in P^1(\mesh;\Rt)\cap H^1_0(\Omega;\Rt)$, cf.~\cite{ScottZ_90_FEI}.
We recall the generic notation $\bn$, $\bt$ for the unit normal and
tangential vectors on any $e\in\cE$, with a certain orientation,
and denote by $[\bt\cdot\tau\bn]$ the jumps of $\bt\cdot\tau\bn$ across the edges.
Then using, in this order, that
$\div_\mesh\tau=0$, $\dual{\trnn(v_h)}{\tau}_\cS=0$,
$[\bt\cdot\tau\bn]=[\bt\cdot(\dev\tau)\bn]$ on any edge $e\in\cE$ by the orthogonality
of the vectors, the approximation properties of the Scott--Zhang interpolation, and
the discrete trace inequality $|e|^{1/2}\|\dev\tau\|_e\lesssim \|\dev\tau\|_\el$
for $e\in\cE(\el)$ and $\el\in\mesh$, we calculate
\begin{align*}
   &\vdual{\tau}{\strain{v}} = \dual{\trnn(v)}{\tau}_\cS
   = \dual{\trnn(v-v_h)}{\tau}_\cS
   = \sum_{e\in\cE(\Omega)} \dual{(v-v_h)\cdot\bt}{[\bt\cdot\tau\bn]}_e\\
   &= \sum_{e\in\cE(\Omega)} \dual{(v-v_h)\cdot\bt}{[\bt\cdot(\dev\tau)\bn]}_e
   \le \sum_{e\in\cE(\Omega)} |e|^{-1/2}\|v-v_h\|_e |e|^{1/2}\|\dev\tau\|_e
   \lesssim \|v\|_1 \|\dev\tau\|.
\end{align*}
Taking the supremum with respect to $v\in H^1_0(\Omega;\Rt)$ with $\|v\|_1=1$ gives
\(
   \|\div\tau\|_{-1} \lesssim \|\dev\tau\|
\)
as wanted. This finishes the proof.
\end{proof}

\section{Mixed finite element scheme} \label{sec_fem}

Let $S^1(\mesh;\Rt):=P^1(\mesh;\Rt)\cap H^1(\Omega;\Rt)$ be
the standard space of continuous, piecewise linear vector-valued functions,
and set $S^1_D(\mesh;\Rt):=S^1(\mesh;\Rt)\cap H^1_D(\Omega;\Rt)$.
For $s>1$, denote by $\Pi^{1,c}_\mesh:\; H^s(\Omega;\Rt)\to S^1(\mesh;\Rt)$
the nodal interpolation operator.
Recall that we are using extensions $u_D\in H^1(\Omega;\Rt)$ and
$\sigma_N\in H(\div,\Omega;\SS)$ of the Dirichlet and Neumann data in \eqref{BC}.
We assume that $u_D\in H^{s}(\Omega;\Rt)$ with $s>1$ so that the interpolant
\begin{equation} \label{Dir_approx}
   u_{D,h}:=\Pi^{1,c}_\mesh(u_D)
\end{equation}
is well defined.
Corresponding to $\sigma_0$ in \eqref{sigma0}, we introduce the constant tensor
\begin{equation} \label{sigma0h}
   \sigma_{0,h}:=
   \frac {\lambda+\mu}{|\Omega|} \dual{u_{D,h}\cdot\bn}{1}_{\G{hc}\cup\G{sc}} \id.
\end{equation}
We select the discrete trace spaces
\[
   S^1(\cS) := \trnn\bigl(S^1(\mesh;\Rt)\bigr),\quad
   S^1_D(\cS) := \trnn\bigl(S^1_D(\mesh;\Rt)\bigr).
\]
Then our mixed finite element scheme reads as follows.

\emph{Find $\sigma_h\in P^2_{nn}(\mesh)$, $u_h\in P^1(\mesh;\Rt)$, and $\eta_h\in S^1(\cS)$
such that}
\begin{subequations} \label{FEM}
\begin{align} \label{BCh}
   &\eta_h|_\G{hc}=u_{D,h}|_\G{hc},\quad
    \eta_h\cdot\bn|_\G{sc}=u_{D,h}\cdot\bn|_\G{sc},\quad
    \eta_h\cdot\bt|_\G{ss}=u_{D,h}\cdot\bt|_\G{ss}
\end{align}
\emph{and, for any $\dsigma\in P^2_{nn}(\mesh)$, $\du\in P^1(\mesh;\Rt)$,
and $\deta\in S^1_D(\cS)$,}
\begin{alignat}{3}
   &\vdual{\AA\sigma_h}{\dsigma} + \vdual{u_h}{\div\dsigma}_\mesh - \dual{\eta_h}{\dsigma}_\cS
   &&= 0,\label{FEMa}\\
   &\vdual{\div\sigma_h}{\du}_\mesh - \dual{\deta}{\sigma_h}_\cS
   &&= -\vdual{f}{\du}-\dual{\deta}{\sigma_N}_\cS\label{FEMb}.
\end{alignat}
\end{subequations}
We note that the degrees of freedom in $P^2_{nn}(\mesh)$ allow for an explicit implementation
of the normal-normal part of the Neumann boundary conditions,
e.g., by using the data of $\Inn{\mesh}(\sigma_N)$.
For ease of presentation we selected the weak form in \eqref{FEMb} for
all of the Neumann conditions, cf.~Remark~\ref{rem_bc}.
We also note that the approximation of Dirichlet data by interpolation
is just one of many choices, cf.~\eqref{Dir_approx},~\eqref{BCh}.
Any piecewise polynomial selection
that is conforming in the trace space of $H^1(\Omega;\Rt)$ is permitted.
The corresponding data approximation error will appear in the error estimates
implicitly through the approximation of $\eta$ by an element
$\rho\in S^1(\cS)$ that satisfies the discrete Dirichlet conditions in \eqref{BCh},
i.e., by an element of the subset
\[
   S^1_{g_{D,h}}(\cS) := \{\rho\in S^1(\cS);\;
   \rho|_\G{hc}=u_{D,h}|_\G{hc},\ \rho\cdot\bn|_\G{sc}=u_{D,h}\cdot\bn|_\G{sc},\
   \rho\cdot\bt|_\G{ss}=u_{D,h}\cdot\bt|_\G{ss}\}.
\]
Our second main result is the locking-free, quasi-optimal convergence of
mixed finite element scheme \eqref{FEM}.

\begin{theorem} \label{thm_Cea}
Let $f\in L_2(\Omega;\Rt)$ and boundary data as in \eqref{ext} be given
where $u_D\in H^s(\Omega)$ with $s>1$.
Furthermore, let boundary decomposition \eqref{split} be so that
Korn's inequality \eqref{Korn} holds. Then, mixed scheme \eqref{FEM} has a unique solution
$(\sigma_h,u_h,\eta_h)$.

To state error estimates let
$(\sigma,u,\eta)\in H_{nn}(\div,\mesh;\SS)\times L_2(\Omega;\Rt)\times H^{1/2}_t(\cS;\Rt)$
denote the solution to \eqref{VF}, and recall definitions \eqref{sigma0} and \eqref{sigma0h}
of $\sigma_0$ and $\sigma_{0,h}$ respectively,
We distinguish between two cases (and note that they overlap
if $\G{ss}\cup\G{sf}=\emptyset$ and $\sigma_0=\sigma_{0,h}$).\\
1. \emph{Case $\G{ss}\cup\G{sf}\not=\emptyset$ or $\sigma_0=\sigma_{0,h}$.}
If the shape-regular mesh $\mesh$ is a refinement of an initial mesh
$\mesh_0$ that has a boundary node
$a\in \G{ss}\cup\G{sf}$ with $\mathrm{dist}(a,\G{hc}\cup\G{sc})>0$,
and which is not a corner of the polygon $\partial\Omega$ with $\mathrm{dist}(a,\G{ss})=0$,
then scheme \eqref{FEM} converges quasi-optimally,
\begin{align} \label{bound_Cea}
   &\|\sigma-\sigma_h\|_{\div,\mesh} + \|u-u_h\| + \|\eta-\eta_h\|_{1/2,t,\cS}\nonumber\\
   &\lesssim
   \inf\bigl\{
   \|\sigma-\tau\|_{\div,\mesh} + \|u-v\| + \|\eta-\rho\|_{1/2,t,\cS};\;
   \tau\in P^2_{nn}(\mesh), v\in P^1(\mesh;\Rt), \rho\in S^1_{g_{D,h}}(\cS) \bigr\}
\end{align}
with a hidden constant that is independent of $\nu$ and $\mesh$.\\
2. \emph{Case $\G{ss}\cup\G{sf}=\emptyset$.} We consider any shape-regular mesh $\mesh$
and use the stress approximation $\overline{\sigma}_h:=\sigma_h+\sigma_0-\sigma_{0,h}$.
Scheme \eqref{FEM} converges quasi-optimally in the sense that
\eqref{bound_Cea} holds with $\sigma_h$ replaced by $\overline{\sigma}_h$.
As before, the hidden constant is independent of $\nu$ and $\mesh$.
\end{theorem}

\begin{proof}
We use the same structure as in the proof of Theorem~\ref{thm_VF}.
The Dirichlet condition is moved to the right-hand side of \eqref{FEMa}
by adding $\dual{\trnn(u_{D,h})}{\dsigma}_\cS$ to both sides, thus
replacing $\eta_h$ with $\wtilde\eta_h:=\eta_h-\trnn(u_{D,h})\in S^1_D(\cS)$.
We verify the standard conditions for mixed methods.
The uniform (in $\nu$ and $\mesh$) boundedness of the bilinear form
and linear functionals is obvious.

{\bf Discrete inf-sup property.}
Let $w\in P^1(\mesh;\Rt)$ and $\rho\in S^1_D(\cS)$ be given with $\rho=\trnn(v)$ for a function
$v\in S^1_D(\mesh;\Rt)$.
For a parameter $\alpha>0$ to be selected, we define $\tau\in P^2_{nn}(\mesh)$ by
assigning degrees of freedom \eqref{dof_el} as follows,
\begin{subequations} \label{dof_IS}
\begin{alignat}{2}
   &\bn\cdot\tau\bn|_e=0                                  &&\forall e\in\cE,\label{dof_ISa}\\
   &\vdual{\tau}{c}=-\vdual{\strain{v}}{c}                &&\forall c\in P^0(\mesh;\SS),\label{dof_ISb}\\
   &\vdual{\div\tau}{q}_\mesh=\alpha\vdual{w-v}{q}\quad   &&\forall q\in P^1(\mesh;\Rt).\label{dof_ISc}
\end{alignat}
\end{subequations}
Note that \eqref{dof_ISc} implies $\div_\mesh\tau=\alpha(w-v)$.
The definition of duality $\dual{\rho}{\tau}_\cS$ and Young's inequality then show that
\begin{align*}
   &\vdual{\div\tau}{w}_\mesh - \dual{\rho}{\tau}_\cS
   =
   \vdual{\div\tau}{w-v}_\mesh - \vdual{\tau}{\strain{v}}
   =
   \alpha\vdual{w-v}{w-v} + \vdual{\strain{v}}{\strain{v}}\\
   &=
   \alpha\bigl(\|w\|^2 -2\vdual{w}{v} + \|v\|^2\bigr) + \|\strain{v}\|^2
   \ge
   \alpha (1-\delta)\|w\|^2 + \alpha(1-\delta^{-1})\|v\|^2 + \|\strain{v}\|^2
\end{align*}
for any $\delta>0$.
Choosing $\delta=1/2$ and using Korn's inequality \eqref{Korn}, we find that
\[
   \vdual{\div\tau}{w}_\mesh - \dual{\rho}{\tau}_\cS
   \ge
   \frac\alpha 2 \|w\|^2 - \alpha\|v\|^2 + \|\strain{v}\|^2
   \ge
   \frac\alpha 2 \|w\|^2 + \frac 12 \|\strain{v}\|^2
   + \Bigl(\frac{C(\Omega,\G{D})^2}2-\alpha\Bigr)\|v\|^2.
\]
Therefore, selecting $\alpha=C(\Omega,\G{D})^2/3$,
we have seen that $\tau\in P^2_{nn}(\mesh)$ satisfies
\[
   \vdual{\div\tau}{w}_\mesh - \dual{\rho}{\tau}_\cS
   \ge \frac 16 \min\bigl\{C(\Omega,\G{D})^2,3\bigr\}\bigl( \|w\|^2 + \|v\|_1^2\bigr).
\]
Noting that $\|v\|_1\ge \|\rho\|_{1/2,t,\cS}$ by definition of the trace norm,
the discrete inf-sup property then follows if we can bound
\[
   \|\tau\|_{\div,\mesh} \lesssim \|w\| + \|v\|_1.
\]
Relation $\div_\mesh\tau=\alpha(w-v)$ implies $\|\div\tau\|_\mesh\lesssim \|w\|+\|v\|$,
and scaling properties show that
\[
   \|\tau\|\lesssim \|\strain{v}\| + \|h_\mesh\|_{L_\infty(\Omega)} (\|w\| + \|v\|)
   \lesssim \|w\| + \|v\|_1.
\]
If $\G{ss}\cup\G{sf}=\emptyset$, as in the continuous case we need to verify that
the inf-sup property also holds for the constrained tensor space,
\begin{align} \label{Pconstrained}
   \wtilde P_{nn}^2(\mesh) := \{\tau\in P_{nn}^2(\mesh);\; \vdual{\tr(\tau)}{1}=0\}.
\end{align}
This can be seen similarly as in the proof of Proposition~\ref{prop_norm}.
Indeed, setting $c:=\id$ in \eqref{dof_ISb}, we calculate (cf.~Remark~\ref{rem_Inn})
\[
   \vdual{\tr(\tau)}{1}=\vdual{\tau}{\id}=-\vdual{\strain{v}}{\id}
   =-\dual{v}{\bn}=-\dual{v}{\bn}_\G{hc}-\dual{v\cdot\bn}{1}_\G{sc}=0
\]
since $v\in H^1_D(\Omega;\Rt)$. We conclude that the constructed tensor
$\tau$ is an element of $\wtilde P_{nn}^2(\mesh)$.

{\bf Discrete coercivity and error estimate.}
We introduce the discrete kernel
\begin{align} \label{kernel}
   K_{0,h}:=\{\tau\in P^2_{nn}(\mesh);\;
   \vdual{\div\tau}{\du}_\mesh=0\ \forall \du\in P^1(\mesh;\Rt),\quad
   \dual{\deta}{\tau}_\cS=0\ \forall\deta\in S^1_D(\cS)\}.
\end{align}
In particular, any $\tau\in K_{0,h}$ satisfies $\div_\mesh\tau=0$ so that
the bilinear form $\vdual{\AA\cdot}{\cdot}$ is coercive on $K_{0,h}$ by
\eqref{coercive}. We conclude that there is a unique solution $(\sigma_h,u_h,\eta_h)$
of \eqref{FEM}.

To show the {\bf robustness of our scheme}, we recall \eqref{coercive_A}
and use Lemma~\ref{la_trdevdiv2} to control the remaining term $\|\tr(\tau)\|$.
To this end we need to identify a closed subspace
$\Sigma\subset L_2(\Omega;\SS)$ with $\id\not\in\Sigma$ and
argue that the coercivity of $\vdual{\AA\cdot}{\cdot}$ on $K_{0,h}\cap\Sigma$
suffices. We consider the two cases separately.

{\bf 1. Case $\G{ss}\cup\G{sf}\not=\emptyset$.}
In that case there exists $v_0\in S^1_D(\mesh_0;\R^2)$ with
$v_0\cdot\bn\ge 0$ on $\G{ss}\cup\G{sf}$ and $v_0\cdot\bn=1$ at $a$.
Since $\mesh$ is a refinement of $\mesh_0$ it follows that $v_0\in S^1_D(\mesh;\Rt)$.
We define
\[
   \Sigma:=\{\tau\in L_2(\Omega;\SS);\; \vdual{\tau}{\strain{v_0}}=0\}.
\]
It is clear that $\Sigma$ is a closed subspace of $L_2(\Omega;\SS)$
and that $\id\not\in\Sigma$ because
\(
   \vdual{\id}{\strain{v_0}}
   = \dual{v_0\cdot\bn}{1}_{\G{ss}\cup\G{sf}} > 0.
\)
The test functions $\du=v_0$ and $\deta=\trnn(v_0)$ in \eqref{kernel} show that
\begin{align*}
     0 = \dual{\trnn(v_0)}{\tau}_\cS -\vdual{\div\tau}{v_0}_\mesh
     = \vdual{\tau}{\strain{v_1}}\quad\forall\tau\in K_{0,h};
\end{align*}
whence $K_{0,h}\subset\Sigma$.
Therefore, Lemma~\ref{la_trdevdiv2} and bound \eqref{coercive_A} reveal that
\[
   \vdual{\AA\tau}{\tau}
   \gtrsim \|\dev\tau\|^2 + \|\tr(\tau)\|^2
   \simeq \|\tau\|_{\div,\mesh}^2
   \quad\forall\tau\in K_{0,h}.
\]
For the current case we have therefore verified the ingredients for the analysis
of mixed finite element schemes,
and one deduces the locking-free quasi-optimal convergence as stated in theorem.
We give some details for the less standard case below.

{\bf 2. Case $\G{ss}\cup\G{sf}=\emptyset$.}
The selection of $\dsigma=\id$ in \eqref{FEMa} shows that
\begin{align*}
   0=\vdual{\AA\sigma_h}{\id} + \vdual{u_h}{\div\id}_\mesh - \dual{\eta_h}{\id}_\cS
    = \frac 1{2(\lambda+\mu)} \vdual{\tr(\sigma_h)}{1}
   - \dual{u_{D,h}\cdot\bn}{1}_{\G{hc}\cup\G{sc}}
\end{align*}
because
\begin{align} \label{dualidh}
   \dual{\eta_h}{\id}_\cS = \dual{\eta_h}{\bn}_\Gamma
   = \dual{\eta_h}{\bn}_\G{hc} + \dual{\eta_h\cdot\bn}{1}_\G{sc}
   = \dual{u_{D,h}}{\bn}_\G{hc} + \dual{u_{D,h}\cdot\bn}{1}_\G{sc}
\end{align}
by \eqref{BCh}. We conclude that
\[
   \frac 1{2|\Omega|}\vdual{\tr(\sigma_h)}{1}\id
   = 
   \frac {\lambda+\mu}{|\Omega|} \dual{u_{D,h}\cdot\bn}{1}_{\G{hc}\cup\G{sc}} \id
   =
   \sigma_{0,h};
\]
that is, $\sigma_{0,h}$ from \eqref{sigma0h} is a datum
with $\vdual{\tr(\sigma_{0,h})}{1}=\vdual{\tr(\sigma_h)}{1}$.
As in the continuous case, we proceed with taking into account that
$\sigma_{0,h}$ may not vanish; some calculations are trivial if $\sigma_{0,h}=0\in\SS$.

We rewrite the mixed scheme by replacing $\sigma_h$ with $\wtilde\sigma_h:=\sigma_h-\sigma_{0,h}$,
giving
\begin{subequations} \label{FEM2}
\begin{alignat}{3}
   &\vdual{\AA\wtilde\sigma_h}{\dsigma} + \vdual{u_h}{\div\dsigma}_\mesh
  - \dual{\wtilde\eta_h}{\dsigma}_\cS
   &&= \dual{\trnn(u_{D,h})}{\dsigma}_\cS,\label{FEMa2}\\
   &\vdual{\div\wtilde\sigma_h}{\du}_\mesh - \dual{\deta}{\wtilde\sigma_h}_\cS
   &&= -\vdual{f}{\du}-\dual{\deta}{\sigma_N}_\cS\label{FEMb2}
\end{alignat}
\end{subequations}
for $\dsigma\in \wtilde P^2_{nn}(\mesh)$ (constrained space \eqref{Pconstrained}),
$\du\in P^1(\mesh;\Rt)$, and $\deta\in S^1_D(\cS)$.
Here, in the case that $\sigma_{0,h}\not=0$, we have used that
$\vdual{\AA\sigma_{0,h}}{\dsigma}=0$ (by the new constraint for $\dsigma$) and
$\dual{\deta}{\sigma_{0,h}}_\cS=0$ due to the homogeneous Dirichlet trace of $\deta$,
cf.~\eqref{dualidh}.

We observe that scheme \eqref{FEM} is equivalent to \eqref{FEM2},
with identical solutions
$(\sigma_h,u_h,\eta_h)=(\wtilde\sigma_h+\sigma_{0,h},u_h,\wtilde\eta_h+\trnn(u_{D,h}))$.
Lemma~\ref{la_trdevdiv2} is applicable with
$\Sigma:=\{\tau\in L_2(\Omega;\SS);\;\vdual{\tr(\tau)}{1}=0\}$,
implying coercivity
\begin{align*}
   \vdual{\AA\tau}{\tau}
   \gtrsim \|\tau\|_{\div,\mesh}^2
   \quad\forall\tau\in K_{0,h}\cap\Sigma.
\end{align*}
This concludes the verification of the ingredients of mixed finite element analysis
and our reformulated scheme \eqref{FEM2} is uniformly well posed and
converges uniformly quasi-optimally. Since non-homogeneous Dirichlet boundary data
for mixed methods are rarely analyzed in the literature, and since our
formulation is non-standard, let us present the details.

We introduce the bilinear form
\[
   \mathcal{B}((\tau,v,\rho);(\dtau,\dv,\drho)):=
   \vdual{\AA\tau}{\dtau}
   + \vdual{v}{\div\dtau}_\mesh +\vdual{\div\tau}{\dv}_\mesh
   - \dual{\rho}{\dtau}_\cS - \dual{\drho}{\tau}_\cS
\]
and spaces
\begin{align*}
   V_h &:=P^{2}_{nn}(\mesh)\times P^1(\mesh;\Rt)\times S^1(\cS)
   &&\hspace{-2em}\subset\
   V:=H_{nn}(\div,\mesh;\SS)\times L_2(\Omega;\Rt)\times H^{1/2}_t(\cS;\Rt),\\
   \wtilde V_h &:=\wtilde P^2_{nn}(\mesh)\times P^1(\mesh;\Rt)\times S^1_D(\cS)
   &&\hspace{-2em}\subset\
   \wtilde V :=\wtilde H_{nn}(\div,\mesh;\SS)\times L_2(\Omega;\Rt)\times H^{1/2}_{D,t}(\cS;\Rt)
\end{align*}
with (squared) norm
\[
   \|(\tau,v,\rho)\|_V^2 := \|\tau\|_{\div,\mesh}^2 + \|v\|^2 + \|\rho\|_{1/2,t,\cS}^2.
\]
(Recall \eqref{Hconstrained} resp. \eqref{Pconstrained}
for the definition of $\wtilde H_{nn}(\div,\mesh;\SS)$ resp. $\wtilde P^2_{nn}(\mesh;\SS)$.)
With the well-posedness of scheme \eqref{FEM2} we have shown the inf-sup property
\[
   \sup_{(\dtau,\dv,\drho)\in \wtilde V_h,\; \|(\dtau,\dv,\drho)\|_V=1}
   \mathcal{B}((\tau,v,\rho);(\dtau,\dv,\drho)) \gtrsim \|(\tau,v,\rho)\|_V\quad
   \forall (\tau,v,\rho)\in \wtilde V_h
\]
uniformly with respect to $\nu$ and $\mesh$.
We select any $(\tau,v,0)\in \wtilde V_h$ and $\rho\in S^1(\cS)$
with $\rho-\eta_h\in S^1_D(\cS)$, and use
formulations \eqref{VF2}, \eqref{FEM2} to conclude that
\begin{equation} \label{infsuph}
\begin{split}
   &\|(\wtilde\sigma_h-\tau,u_h-v,\eta_h-\rho)\|_V
   \lesssim \sup_{0\not=(\dsigma,\du,\deta)\in \wtilde V_h}
            \frac{\mathcal{B}((\wtilde\sigma_h-\tau,u_h-v,\eta_h-\rho);(\dsigma,\du,\deta))}
                 {\|(\dsigma,\du,\deta)\|_V}\\
   &\qquad=
   \sup_{0\not=(\dsigma,\du,\deta)\in \wtilde V_h}
   \frac{\mathcal{B}((\wtilde\sigma-\tau,u-v,\eta-\rho);(\dsigma,\du,\deta))}
        {\|(\dsigma,\du,\deta)\|_V}
   \lesssim
   \|(\wtilde\sigma-\tau,u-v,\eta-\rho)\|_V
\end{split}
\end{equation}
holds.
We define $\overline{\sigma}_h:=\sigma_h+\sigma_0-\sigma_{0,h}$ with $\sigma_0$ from
\eqref{sigma0} and note that
$\overline{\sigma}_h\in V_h$ and $\sigma-\overline{\sigma}_h=\wtilde\sigma-\wtilde\sigma_h$.
Then, bound \eqref{infsuph} and an application of the triangle inequality show that
\begin{align*}
   &\|\sigma-\overline{\sigma}_h\|_{\div,\mesh}
   + \|u-u_h\|
   + \|\eta-\eta_h\|_{1/2,t,\cS}
   =
   \|\wtilde\sigma-\wtilde\sigma_h\|_{\div,\mesh}
   + \|u-u_h\|
   + \|\eta-\eta_h\|_{1/2,t,\cS}\\
   &\lesssim
   \|\wtilde\sigma-\tau\|_{\div,\mesh} + \|u-v\| + \|\eta-\rho\|_{1/2,t,\cS}
   =
   \|\sigma-(\tau+\sigma_0)\|_{\div,\mesh} + \|u-v\| + \|\eta-\rho\|_{1/2,t,\cS}
\end{align*}
for any $(\tau,v,0)\in \wtilde V_h$ and $\rho\in S^1(\cS)$ with $\rho-\eta_h\in S^1_D(\cS)$.
The trace constraint $\vdual{\tr(\wtilde\tau)}{1}=\vdual{\tr(\sigma)}{1}$
(which $\wtilde\tau=\tau+\sigma_0$ satisfies)
is achieved by adding an appropriate multiple of tensor $\id$ to any
$\wtilde\tau\in P^2_{nn}(\mesh)$.
This only affects the $L_2$-norm $\|\sigma-\wtilde\tau\|$ so that
\begin{align*}
   \inf_{\tau\in \wtilde P^2_{nn}(\mesh)}
   \|\sigma-(\tau+\sigma_0)\|_{\div,\mesh}
   \le
   \inf_{\tau\in P^2_{nn}(\mesh)} \|\sigma-\tau\|_{\div,\mesh},
\end{align*}
and we can use this upper bound for the error estimate.
This finishes the proof of the theorem.
\end{proof}

It remains to show appropriate convergence orders of the mixed scheme.
In order to derive an optimal order for the displacement approximation
we need to specify a regularity shift.
Specifically, if $u_0\in H^1(\Omega;\Rt)$ solves
$-\div\CC\strain{u_0}=g$ in $\Omega$ with $g\in L_2(\Omega)$ and
homogeneous Dirichlet and Neumann boundary data according
to boundary decomposition \eqref{split}, we assume that
\begin{equation} \label{sreg}
   u_0\in H^{1+\sreg}(\Omega;\Rt) 
   \hspace{0.5em}\text{with}\hspace{0.5em}
   \|u_0\|_{1+\sreg} \le \creg \|g\|
\end{equation}
where $\creg>0$ and $0<\sreg\le 1$ are independent of $\nu$ and $g$.
For convex polygonal domains and homogeneous (hard clamped) Dirichlet
condition, $\sreg=1$ can be taken, whereas for general homogenous
boundary conditions and non-convex domains there is only a reduced shift $\sreg<1$.
For details we refer to \cite{Grisvard_89_SEE,SaendigRS_89_RBV,Roessle_00_CSR}.


\begin{theorem} \label{thm_FEM}
Given $0<s\le 1$, $s\le r\le 2$, and $f\in H^r(\Omega;\Rt)$, we assume that the
boundary data in \eqref{BC} are sufficiently smooth
so that the solution $u$ of \eqref{model} satisfies $u\in H^{1+s}(\Omega;\Rt)$.
Let $(\sigma,u,\eta)$ and $(\sigma_h,u_h,\eta_h)$ be the solutions of
\eqref{VF} and \eqref{FEM}, respectively, and recall that
$\lambda>0$ is one of the Lam\'e constants.
We assume that the assumptions of Theorem~\ref{thm_Cea} hold and denote
$\sigma_h^*:=\sigma_h$ if $\G{ss}\cup\G{sf}\not=\emptyset$ and, otherwise,
$\sigma_h^*:=\overline{\sigma}_h=\sigma_h+\sigma_0-\sigma_{0,h}$.
The error estimates
\begin{align*}
   \|\sigma-\sigma_h^*\|_{\div,\mesh} + \|u-u_h\| + \|\trnn(u)-\eta_h\|_{1/2,t,\cS}
   &\lesssim
   \|h_\mesh\|_{\infty}^s (\|u\|_{1+s} + \|\sigma\|_s + \|f\|_s),\\
   \|\div(\sigma-\sigma_h^*)\|_\mesh =
   \|\div(\sigma-\sigma_h)\|_\mesh &\lesssim \|h_\mesh\|_{\infty}^r \|f\|_r
\end{align*}
hold true.
Furthermore, assuming that \eqref{sreg} holds and the Dirichlet data vanish,
we have the estimate
\[
   \|u-u_h\| \lesssim
   \creg \|h_\mesh\|_\infty^{s+\sreg} \bigl(\|u\|_{1+s} + \|f\|_s\bigr).
\]
All the intrinsic constants appearing in the error estimates are independent of
$\nu$ and $\mesh$ (under the conditions specified in Theorem~\ref{thm_Cea}).
\end{theorem}

\begin{proof}
By Theorem~\ref{thm_Cea} we know that scheme \eqref{FEM} is well posed and converges
quasi-optimally, as specified in the cases~1 and 2.
The bound for $\|\div(\sigma-\sigma_h)\|_\mesh$ follows from relations
\eqref{FEMb} with $\deta=0$. (Note that $\div\sigma_h^*=\div\sigma_h$ because
$\div\sigma_0=\div\sigma_{0,h}=0$.)
The estimate for the total error follows by standard approximation results and
Proposition~\ref{prop_Inn}.

{\bf Error estimate for $u$.}
For brevity, let us denote $h:=\|h_\mesh\|_\infty$.
By assumption, the Dirichlet data vanish (implying that $\sigma_h^*=\sigma_h$).
We apply the Aubin--Nitsche technique for mixed formulations, cf.~\cite{DouglasR_85_GEM},
and split $u-u_h=v+w$ with $v:=u-\Pi^1_\mesh u$ and $w:=\Pi^1_\mesh u-u_h$.
Note that $\vdual{v}{\div\dsigma}_\mesh=0$ for $\dsigma\in P^2_{nn}(\mesh)$
since $\div_\mesh\dsigma\in P^1(\mesh)$. Therefore, denoting
\[
   \tau:=\sigma-\sigma_h\in H_{nn}(\div,\mesh;\SS)\quad\text{and}\quad
   \rho:=\eta-\eta_h\in H^{1/2}_{D,t}(\cS;\Rt),
\]
subtraction of \eqref{VF} and \eqref{FEM} shows that
$w=\Pi^1_\mesh u-u_h\in P^1(\mesh;\Rt)$ satisfies
\begin{subequations} \label{probw}
\begin{alignat}{3}
   &\vdual{\AA\tau}{\dsigma} + \vdual{w}{\div\dsigma}_\mesh
   - \dual{\rho}{\dsigma}_\cS
   &&= -\vdual{v}{\div\dsigma}_\mesh=0, \label{probw_a}\\
   &\vdual{\div\tau}{\du}_\mesh - \dual{\deta}{\tau}_\cS
   &&=0 
   \label{probw_b}
\end{alignat}
\end{subequations}
for any $\dsigma\in P^2_{nn}(\mesh)$, $\du\in P^1(\mesh;\Rt)$, and $\deta\in S^1_D(\cS)$.
Our aim is to bound $\|w\|$.
To this end let $u_w\in H^1(\Omega;\Rt)$ solve $\div\CC\strain{u_w}=w$
in $\Omega$ with homogeneous Dirichlet and Neumann conditions according to
decomposition \eqref{split} of $\Gamma$.
Owing to \eqref{sreg}, $u_w\in H^{1+\sreg}(\Omega;\Rt)\cap H^1_D(\Omega;\SS)$,
$\sigma_w:=\CC\strain{u_w}\in H^{\sreg,0}(\div,\Omega;\SS)\cap H_N(\div,\Omega;\SS)$, and
$\eta_w:=\trnn(u_w)\in\trnn\bigl(H^{1+\sreg}(\Omega;\Rt)\bigr)\cap H^{1/2}_{D,t}(\cS;\Rt)$.
By representation
$\sigma_w=2\mu\strain{u_w}+\lambda(\div u_w) \id$,
$\dev\sigma_w=2\mu(\strain{u_w}-\div u_w\id/2)$, so that
Lemma~\ref{la_trdevdiv} and \eqref{sreg} reveal that
\begin{align} \label{b_sigmaw}
   \|\sigma_w\|_\sreg \lesssim \|\dev\sigma_w\|_\sreg + \|\tr(\sigma_w)\|_\sreg
   \lesssim \|u_w\|_{1+\sreg} + \|w\|_{\sreg-1} \lesssim \|w\|
\end{align}
holds with a hidden constant that is independent of $w$ and $\nu$.
For what follows we set
$\wtilde u_w:=\Pi^{1,c}_\mesh u_w\in S^1_D(\mesh;\Rt)$,
$\wtilde\sigma_w:=\Inn{\mesh}\sigma_w\in P^2_{nn}(\mesh)$,
and $\wtilde\eta_w:=\trnn(\wtilde u_w)\in S^1_D(\cS)$.
Relations \eqref{probw_b} and an application of trace operator $\trnn$ show that
\begin{align*}
  0 &= -\vdual{\div\tau}{\wtilde u_w}_\mesh
     = -\vdual{\div\tau}{u_w}_\mesh + \vdual{\div\tau}{u_w-\wtilde u_w}_\mesh\\
    &=  \vdual{\AA\tau}{\CC\strain{u_w}} - \dual{\trnn(u_w)}{\tau}_\cS
       + \vdual{\div\tau}{u_w-\wtilde u_w}_\mesh\\
    &=  \vdual{\AA\tau}{\sigma_w-\wtilde\sigma_w} + \vdual{\AA\tau}{\wtilde\sigma_w}
       -\dual{\eta_w-\wtilde\eta_w}{\tau}_\cS + \vdual{\div\tau}{u_w-\wtilde u_w}_\mesh
\end{align*}
(note that $\dual{\wtilde\eta_w}{\tau}_\cS=0$ by \eqref{probw_b}),
that is,
\[
   -\vdual{\AA\tau}{\wtilde\sigma_w}
   = \vdual{\AA\tau}{\sigma_w-\wtilde\sigma_w} -\dual{\eta_w-\wtilde\eta_w}{\tau}_\cS
   + \vdual{\div\tau}{u_w-\wtilde u_w}_\mesh.
\]
Therefore, recalling that $w\in P^1(\mesh;\Rt)$, using the commutativity property of $\Inn{\mesh}$
(Proposition~\ref{prop_Inn}), relations \eqref{probw},
and noting that $\dual{\rho}{\sigma_w}_\cS=0$ by the chosen boundary conditions
(cf.~Lemma~\ref{la_Inn_cons}), we find that
\begin{align} \label{bound_w}
      \|w\|^2
   &= \vdual{w}{\div \sigma_w}
   = \vdual{w}{\div\wtilde\sigma_w}_\mesh
   = -\vdual{\AA\tau}{\wtilde\sigma_w} + \dual{\rho}{\wtilde\sigma_w}_\cS \nonumber\\
   &= \vdual{\AA\tau}{\sigma_w-\wtilde\sigma_w} -\dual{\eta_w-\wtilde\eta_w}{\tau}_\cS
   + \vdual{\div\tau}{u_w-\wtilde u_w}_\mesh
   - \dual{\rho}{\sigma_w-\wtilde\sigma_w}_\cS.
\end{align}
To continue bounding \eqref{bound_w} we need some estimates. The following ones
are standard or due to Proposition~\ref{prop_Inn},
bound \eqref{b_sigmaw} and regularity shift \eqref{sreg},
\begin{subequations} \label{err_std}
\begin{align}
   &\|\sigma_w-\wtilde\sigma_w\|=\|\sigma_w-\Inn{\mesh}\sigma_w\|
    \lesssim h^\sreg\|\sigma_w\|_\sreg\lesssim h^\sreg \|w\|,\\
   &\|\eta_w-\wtilde\eta_w\|_{1/2,t,\cS}=\|\trnn(u_w-\Pi^{1,c}_\mesh u_w)\|_{1/2,t,\cS}
   \lesssim h^\sreg\|u_w\|_{1+\sreg} \lesssim h^\sreg \|w\|,\\
   &\|u_w-\wtilde u_w\| \lesssim h^{1+\sreg} \|u_w\|_{1+\sreg} \lesssim h^{1+\sreg} \|w\|.
\end{align}
\end{subequations}
To estimate $\dual{\rho}{\sigma_w-\wtilde\sigma_w}_\cS$ we follow \cite{Fuehrer_18_SDM}
(see Theorem~5 there) and consider
any $v_\rho\in H^1(\Omega;\Rt)$ with $\trnn(v_\rho)=\rho$ to calculate
\begin{align*}
   \dual{\rho}{\sigma_w-\wtilde\sigma_w}_\cS
   &= \vdual{\strain{v_\rho}}{\sigma_w-\wtilde\sigma_w}
   + \vdual{v_\rho}{\div(\sigma_w-\wtilde\sigma_w)}_\mesh\\
   &= \vdual{\strain{v_\rho}}{\sigma_w-\Inn{\mesh}\sigma_w}
   + \vdual{v_\rho-\Pi^1_\mesh v_\rho}{(1-\Pi^1_\mesh)\div\sigma_w}_\mesh.
\end{align*}
Consequently,
\[
   \left|\dual{\rho}{\sigma_w-\wtilde\sigma_w}_\cS\right|
   \lesssim h^\sreg \|v_\rho\|_1 \|\sigma_w\|_\sreg + h \|v_\rho\|_1 \|\div\sigma_w\|
   \lesssim \creg h^\sreg \|v_\rho\|_1 \|w\|.
\]
Taking the infimum with respect to $v_\rho$ subject to $\trnn(v_\rho)=\rho$ we
have shown that
\begin{align} \label{bound_rho}
   \left|\dual{\rho}{\sigma_w-\wtilde\sigma_w}_\cS\right|
   \lesssim \creg h^\sreg \|\rho\|_{1/2,t,\cS} \|w\|.
\end{align}
A combination of \eqref{bound_w}, \eqref{err_std}, \eqref{bound_rho}
and the Cauchy--Schwarz inequality reveals that
\begin{align*}
   \|w\|^2
   &=
   \left|\vdual{\AA\tau}{\sigma_w-\wtilde\sigma_w} -\dual{\eta_w-\wtilde\eta_w}{\tau}_\cS
          +\vdual{\div\tau}{u_w-\wtilde u_w}_\mesh
          - \dual{\rho}{\sigma_w-\wtilde\sigma_w}_\cS
    \right|
   \\
   &\lesssim
   \creg h^{s'} \|w\| \Bigl(\|\tau\|_{\div,\mesh} + \|\rho\|_{1/2,t,\cS}\Bigr).
\end{align*}
Using the previously derived error estimate
for $\tau=\sigma-\sigma_h$ and $\rho=\eta-\eta_h$, we have shown that
\[
   \|w\| \lesssim \creg h^\sreg (\|\tau\|_{\div,\mesh} + \|\rho\|_{1/2,t,\cS})
         \lesssim \creg h^{s+\sreg} (\|u\|_{1+s} + \|f\|_s).
\]
The desired bound follows by an application of the triangle inequality.
This finishes the proof.
\end{proof}

\begin{remark}
Recall that the optimal-order error estimate for the displacement requires homogeneous Dirichlet
data. In fact, the appropriate treatment of non-homogeneous Dirichlet conditions
can be a delicate issue, cf.~\cite{BartelsCD_04_IDC}.
A projection-based approximation of the Dirichlet data is indicated
to derive an optimal-order displacement estimate in this case as well.
\end{remark}

\section{Numerical experiments} \label{sec_num}

We start with a comment about discrete trace space
\(
   S^1(\cS) = \trnn\Bigl(S^1(\mesh;\Rt)\Bigr)
\)
and the implementation of the terms $\dual{\rho}{\tau}_\cS$
for discrete functions $\rho\in S^1(\cS)$ and $\tau\in P^2_{nn}(\mesh)$.
Considering an individual triangle $\el\in\mesh$, it is easy to see that
there is a bijection between linear polynomials $v\in P^1(\el;\R)$ and
their tangential traces on $\partial\el$.
(Of course, there is also a bijection with the full traces.)
Therefore, we can identify trace space $S^1(\cS)$ with the standard space of nodal functions
$S^1(\mesh;\Rt)$. Then, recalling Remark~\ref{rem_trace}, we simply calculate
\[
   \dual{\trnn(v)}{\tau}_\cS
   = \sum_{\el\in\mesh} \sum_{e\in\cE(\el)} \dual{v}{\tau\bn}_{e}
   \quad (v\in S^1(\mesh;\Rt),\ \tau\in P^2_{nn}(\mesh)).
\]
Integrals $\vdual{f}{\du}$ and $\dual{\trnn(v)}{\sigma_{N}}_\cS=\dual{v}{\sigma_N}_\Gamma$
(cf.~Lemma~\ref{la_Inn_cons}) are calculated by Gauss formulas with three and two points
on elements and edges, respectively, whereas for error calculation we use a seven-point formula.
For all the examples we use uniform meshes consisting of $N$ triangles with diameter $h$.

Writing strain and stress as vectors in $\R^3$, the elasticity and compliance tensors are
\[
   \CC = \begin{pmatrix}
            2\mu+\lambda & 0 & \lambda\\ 0 & 2\mu & 0\\ \lambda & 0 & 2\mu+\lambda
         \end{pmatrix},\quad
   \AA = \frac 1{4\mu}
         \begin{pmatrix}
            1+\frac \mu{\mu+\lambda} & 0 & -\frac \lambda{\mu+\lambda}\\
            0 & 2 & 0\\
            -\frac \lambda{\mu+\lambda} & 0 & 1+\frac \mu{\mu+\lambda}
         \end{pmatrix}
\]
with Lam\'e constants
$\lambda=\frac {E\nu}{(1+\nu)(1-2\nu)}$, $\mu=\frac E{2(1+\nu)}$
for Young's modulus $E$ and Poisson ratio $\nu$. Throughout we select $E=1$.

\subsection{Smooth example in unit square} \label{sec_num_reg1}

We use the domain $\Omega=(0,1)^2$ and consider boundary decomposition \eqref{split} with
\[
   \G{hc}=(0,1)\times\{0\},\quad
   \G{sc}=\{1\}\times (0,1),\quad
   \G{ss}=(0,1)\times\{1\},\quad
   \G{sf}=\{0\}\times (0,1),
\]
assigning the four types of boundary conditions accordingly. We select the boundary data
and load $f$ so that
\[
   u(x_1,x_2)
   = \begin{pmatrix}
      \sin(3x_1)\cos(3x_2)\\ \cos(3x_1)\sin(3x_2)
     \end{pmatrix}
\]
is the solution of model problem \eqref{model} with $\sigma=\CC\strain{u}$.
One finds that none of the boundary data functions vanishes, cf.~\eqref{BC}.
Only the normal component of $\g{D}$ and tangential component of $\g{N}$ are zero.
This example fits case~1 of Theorem~\ref{thm_Cea}, and no correction of
stress approximation $\sigma_h$ is needed.

The corresponding approximation errors
$\|\sigma-\sigma_h\|$, $\|\div(\sigma-\sigma_h)\|_\mesh$, $\|u-u_h\|$,
normalized by the norm of the exact solution
$\bigl(\|u\|_1^2+\|\sigma\|_{\div}^2\bigr)^{1/2}$,
are shown in Figure~\ref{fig_err_reg1}, for $\nu=0.3$ on the left and for
$\nu=0.4999999$ on the right.
Recalling that trace approximation $\eta_h$ gives rise to a unique 
$\wtilde u_h\in P^1(\mesh;\R^2)\cap H^1(\Omega;\R^2)$ with $\trnn(\wtilde u_h)=\eta_h$,
we use $\strain{\wtilde u_h}$ as piecewise constant approximation of the strain
to calculate $|u-\wtilde u_h|_1=\|\strain{u}-\strain{\wtilde u_h}\|$,
also shown with normalization in the figure.
Up to the $L_2$ contribution $\|u-\wtilde u_h\|$,
$|u-\wtilde u_h|_1$ is an upper bound for $\|\trnn(u)-\eta_h\|_{1/2,t,\cS}$.
By Theorem~\ref{thm_FEM} we expect the overall convergence order $O(h)=O(N^{-1/2})$
and $\|\div(\sigma-\sigma_h)\|_\mesh=O(h^2)=O(N^{-1})$.
This is indeed confirmed by the shown results,
also indicating optimal convergence order $\|u-u_h\|=O(h^2)$
(only proved for the case of homogeneous Dirichlet data).
Furthermore, the results confirm that our scheme is locking free.
Note that $\tr(\sigma)=\tr(\CC\strain{u})=12(\mu+\lambda)\cos(3x_1)\cos(3x_2)$ and
$\dev\sigma=-6\mu\begin{pmatrix} 0 & \sin(3x_1)\sin(3x_2)\\ \sin(3x_1)\sin(3x_2) & 0\end{pmatrix}$
so that
$\|\tr(\sigma)\|$ and $\vdual{\tr(\sigma)}{1}$ tend to infinity when $\nu\to 1/2$,
while $\|\dev\sigma\|=O(1)$.
In other words, this is a suitable example to test for locking, cf.~\eqref{coercive_A}.


\begin{figure}[hb]
\hspace{-2em}
\includegraphics[width=0.55\textwidth]{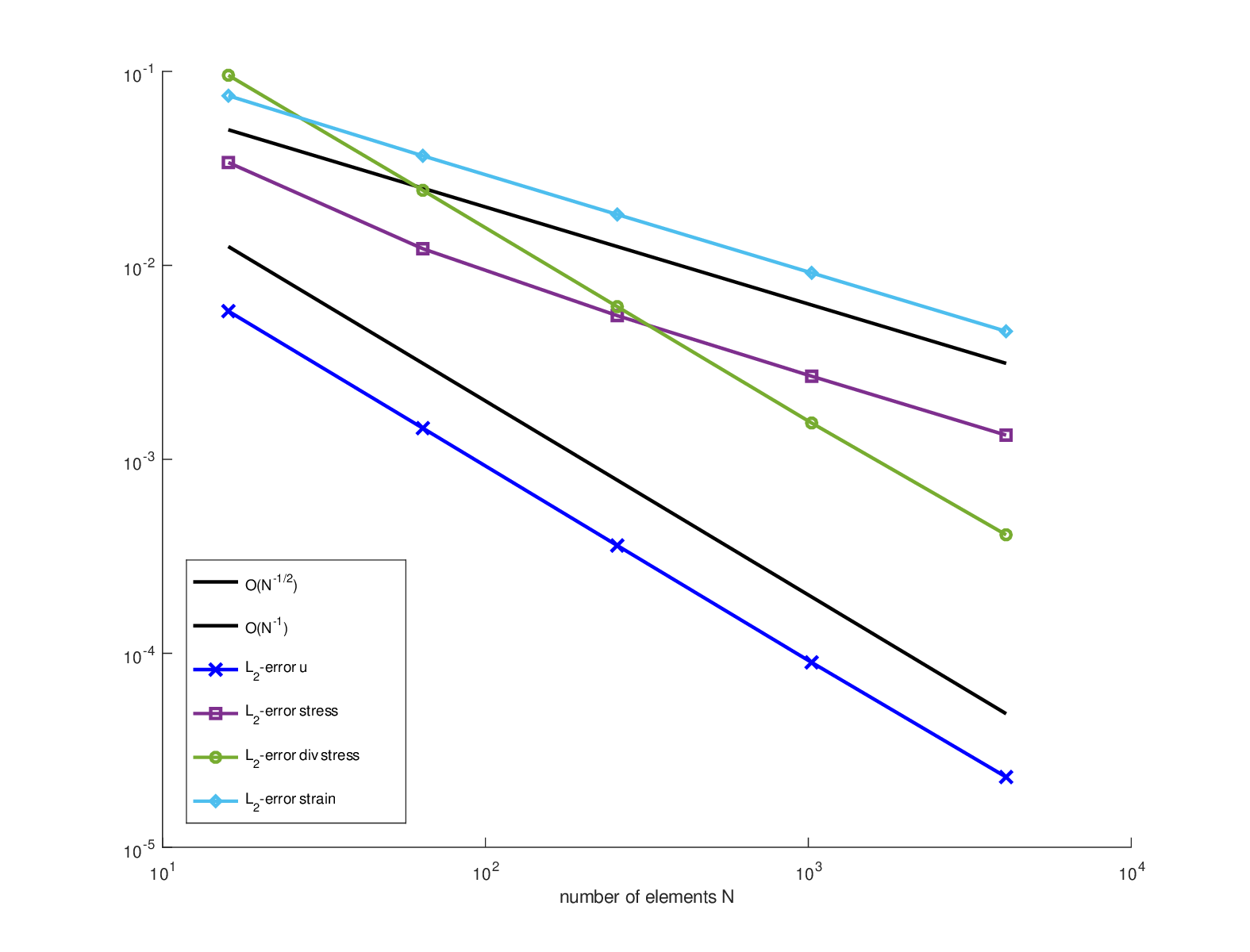}
\nobreak\hspace{-2.5em}
\includegraphics[width=0.55\textwidth]{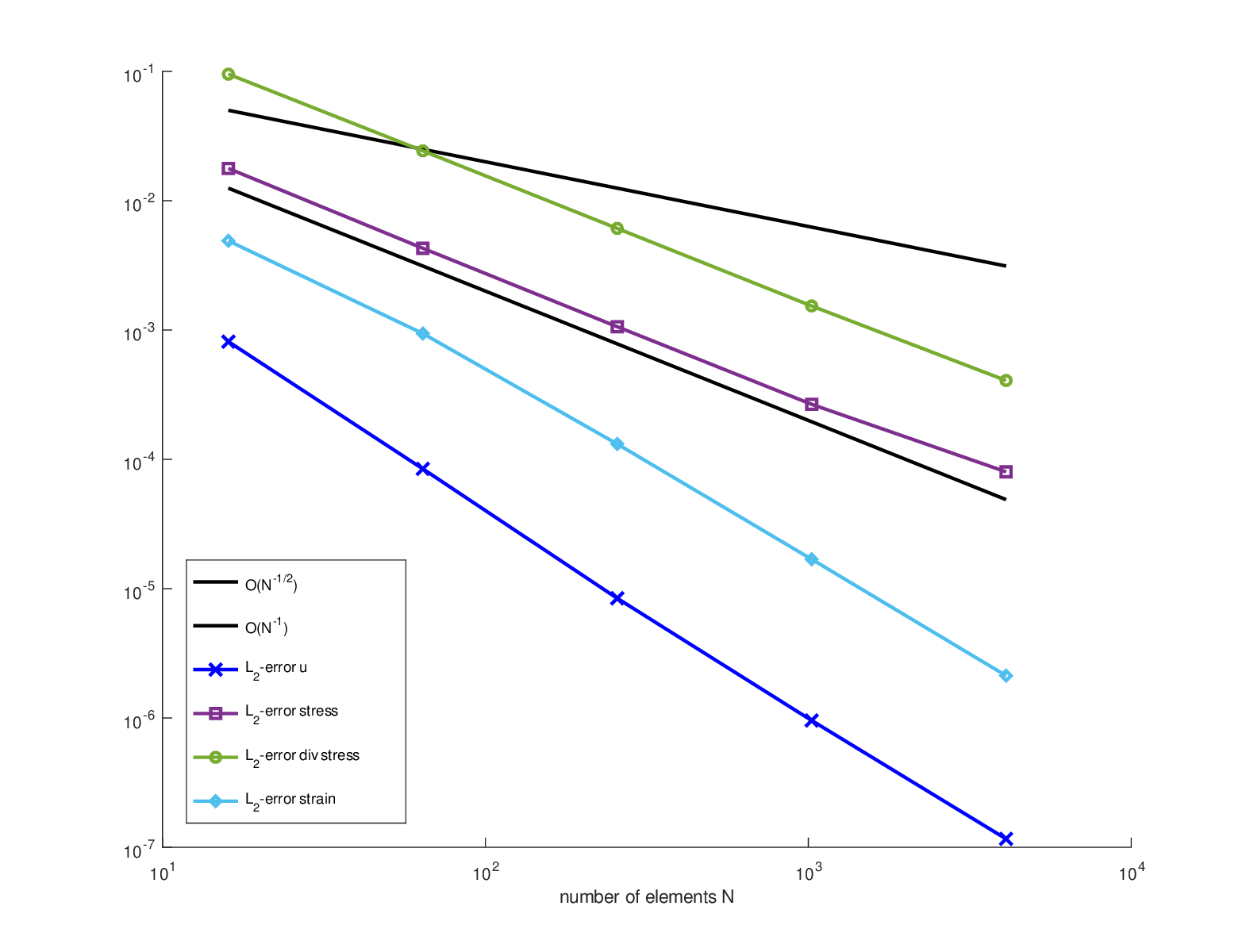}
\caption{Relative errors for the smooth example from \S\ref{sec_num_reg1},
         $\nu=0.3$ (left) and $\nu=0.4999999$ (right).}
\label{fig_err_reg1}
\end{figure}

\subsection{Singular example in L-shaped domain} \label{sec_num_sing1}

Considering the L-shaped domain $\Omega=(-1,1)^2\setminus [0,1)\times (-1,0]$,
we use the following singularity displacement from \cite{SaendigRS_89_RBV},
\[
   u(r,\phi) = r^\alpha \Bigl( C_1(\alpha)v_1(\alpha,\phi) + C_2(\alpha)v_2(\alpha,\phi)\Bigr)
\]
where, for $A(\alpha) := \frac {2(\lambda+3\mu)}{(\lambda+\mu)\alpha}$,
\begin{align*}
   v_1(\alpha,\phi)
   &= \begin{pmatrix}  -\cos(\alpha\phi) + \cos((\alpha-2)\phi)\\
                     \bigl(1-A(\alpha)\bigr)\sin(\alpha\phi) - \sin((\alpha-2)\phi)
     \end{pmatrix},\\
   v_2(\alpha,\phi)
   &= \begin{pmatrix}  -\bigl(1+A(\alpha)\bigr)\sin(\alpha\phi) + \sin((\alpha-2)\phi)\\
                     -\cos(\alpha\phi) + \cos((\alpha-2)\phi)
     \end{pmatrix},\\
   C_1(\alpha) &= -\cos(\alpha\omega) + \cos((\alpha-2)\omega),\quad
   C_2(\alpha) = -\bigl(1-A(\alpha)\bigr)\sin(\alpha\omega) + \sin((\alpha-2)\omega).
\end{align*}
Here, $(r,\phi)$ are the polar coordinates centered at the origin, and
$\alpha$ is a root of $\sin^2(\alpha\omega)-\frac{\alpha^2G^2}{(G-2)^2}\sin^2\omega$
with $G=-(\lambda+\mu)/\mu$. We select $\nu=0.3$.
The interior angle at the origin is $\omega=3\pi/2$, giving $\alpha\approx 0.59516$.
It follows that $u\in H^{1+s}(\Omega)$ and $\sigma=\CC\epsilon(u)\in H^s(\Omega)$
for any $s<\alpha$ solve \eqref{PDE} with $f=0$.
Furthermore, $u=0$ along the edges of $\Omega$ that meet at the origin
and on the remaining edges we use the Dirichlet data imposed by the given solution,
that is, $\G{hc}=\Gamma$.

Figure~\ref{fig_err_sing1} shows the relative errors for a sequence of uniformly
refined meshes. Since $\nu$ is not close to $1/2$ we use the standard stress
approximation $\sigma_h$ rather than $\overline{\sigma}_h$ as indicated in part~2
of Theorem~\ref{thm_Cea}.
The right-hand side function $f$ vanishes so that $\div_\mesh\sigma_h=0$
up to machine precision and, therefore, the error $\|\div(\sigma-\sigma_h)\|_\mesh=0$
is not plotted.
As before, we use $\eta_h$ to recover a continuous displacement approximation $\wtilde u_h$
and corresponding strain $\epsilon(\wtilde u_h)$ to plot an upper bound
(up to the $L_2$ contribution) of $\|\trnn(u)-\eta_h\|_{1/2,t,\cS}$.
By the given regularity we expect the total error to be of order $O(h^\alpha)=O(N^{-\alpha/2})$,
and $\|u-u_h\|=O(h^{2\alpha})=O(N^{-\alpha})$.
This is confirmed by the results.

\begin{figure}[hb]
\begin{center}
\includegraphics[width=0.7\textwidth]{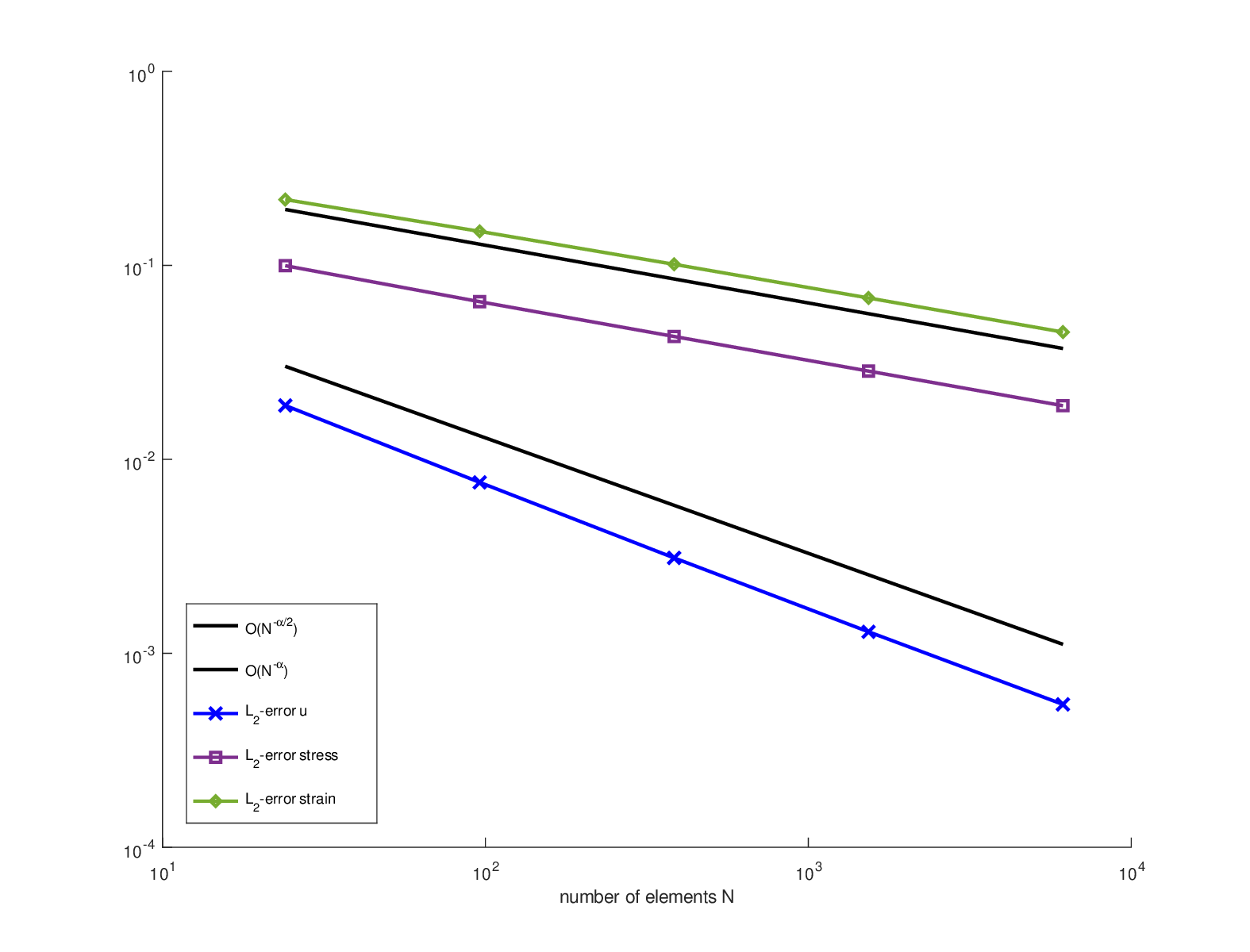}
\end{center}
\caption{Relative errors for the singular example from \S\ref{sec_num_sing1}.}
\label{fig_err_sing1}
\end{figure}

For illustration, Figures~\ref{fig_u}--\ref{fig_stress} show the displacement and stress
approximations on a mesh of 96 elements.

\begin{figure}[hb]
\begin{center}
\begin{subfigure}{\textwidth}
\includegraphics[width=\textwidth]{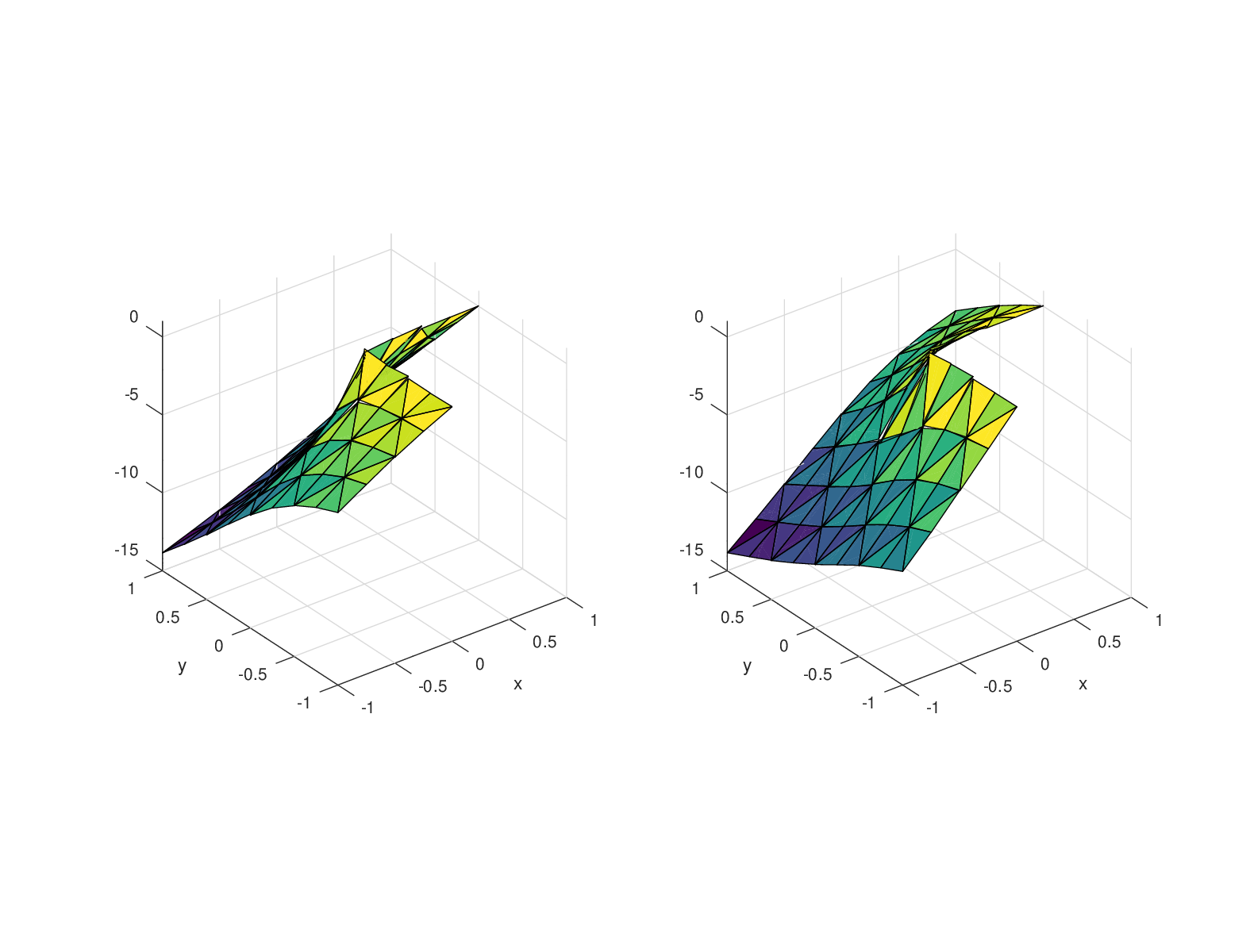}
\vspace{-7em}
\caption{Approximation $u_h$ (discontinuous).}
\end{subfigure}
\begin{subfigure}{\textwidth}
\includegraphics[width=\textwidth]{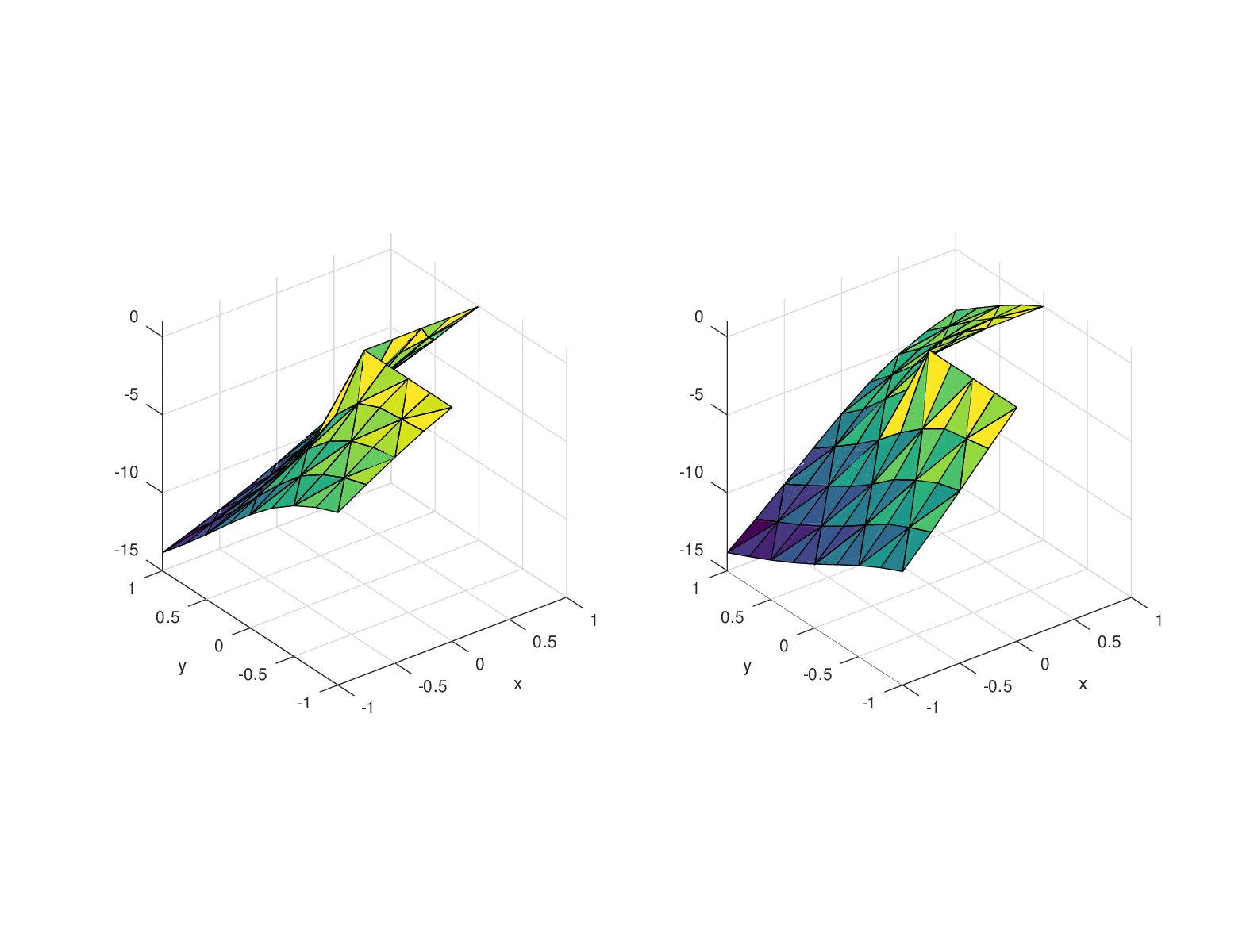}
\vspace{-7em}
\caption{Approximation recovered from $\eta_h$ (continuous).}
\end{subfigure}
\end{center}

\caption{Displacement approximation on a mesh of 96 elements, $u_1$ (left) and $u_2$ (right),
         for the singular example from \S\ref{sec_num_sing1}.}
\label{fig_u}
\end{figure}

\begin{figure}[hb]
\begin{center}
\includegraphics[width=\textwidth]{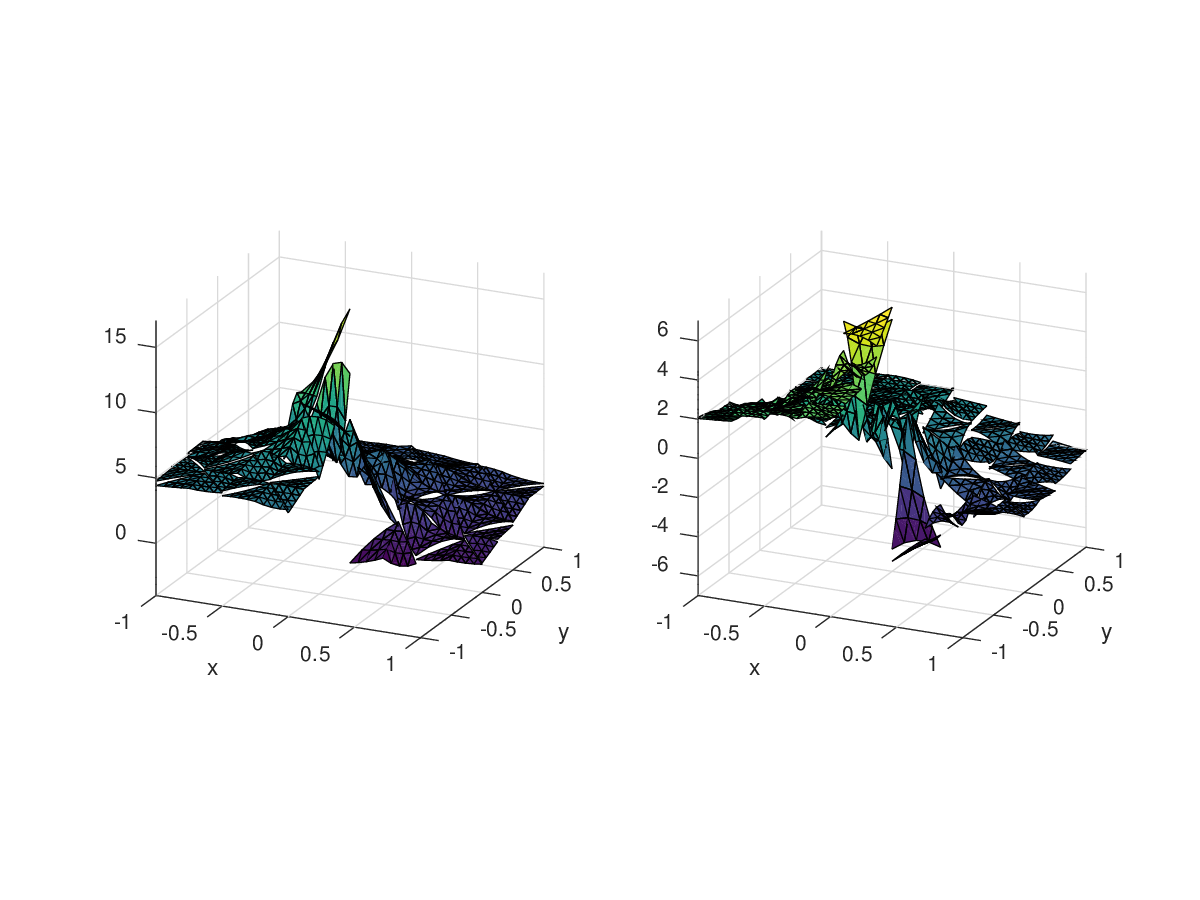}
\end{center}
\caption{Stress approximation on a mesh of 96 elements for the singular
         example from \S\ref{sec_num_sing1},
         $\sigma_{11}$ (left) and $\sigma_{12}$ (right).}
\label{fig_stress}
\end{figure}

\bibliographystyle{plain}
\bibliography{/home/norbert/tex/bib/bib,/home/norbert/tex/bib/heuer}

\begin{thebibliography}{10}

\bibitem{AdamsC_05_MFE}
Scot Adams and Bernardo Cockburn.
\newblock A mixed finite element method for elasticity in three dimensions.
\newblock {\em J. Sci. Comput.}, 25(3):515--521, 2005.

\bibitem{AmroucheBDG_98_VPT}
Ch\'erif Amrouche, Christine Bernardi, Monique Dauge, and Vivette Girault.
\newblock Vector potentials in three-dimensional non-smooth domains.
\newblock {\em Math. Methods Appl. Sci.}, 21(9):823--864, 1998.

\bibitem{ArnoldAW_08_FES}
Douglas~N. Arnold, Gerard Awanou, and Ragnar Winther.
\newblock Finite elements for symmetric tensors in three dimensions.
\newblock {\em Math. Comp.}, 77(263):1229--1251, 2008.

\bibitem{ArnoldW_02_MFE}
Douglas~N. Arnold and Ragnar Winther.
\newblock Mixed finite elements for elasticity.
\newblock {\em Numer. Math.}, 92(3):401--419, 2002.

\bibitem{BartelsCD_04_IDC}
S\"{o}ren Bartels, Carsten Carstensen, and Georg Dolzmann.
\newblock Inhomogeneous {Dirichlet} conditions in a priori and a posteriori
  finite element error analysis.
\newblock {\em Numer. Math.}, 99:1--24, 2004.

\bibitem{BoffiBF_13_MFE}
Daniele Boffi, Franco Brezzi, and Michel Fortin.
\newblock {\em Mixed finite element methods and applications}, volume~44 of
  {\em Springer Series in Computational Mathematics}.
\newblock Springer, Heidelberg, 2013.

\bibitem{BottassoMS_02_DPG}
Carlo~L. Bottasso, Stefano Micheletti, and Riccardo Sacco.
\newblock The discontinuous {P}etrov-{G}alerkin method for elliptic problems.
\newblock {\em Comput. Methods Appl. Mech. Engrg.}, 191(31):3391--3409, 2002.

\bibitem{CarstensenDG_16_BSF}
Carsten Carstensen, Leszek~F. Demkowicz, and Jay Gopalakrishnan.
\newblock Breaking spaces and forms for the {DPG} method and applications
  including {M}axwell equations.
\newblock {\em Comput. Math. Appl.}, 72(3):494--522, 2016.

\bibitem{CarstensenD_98_PEE}
Carsten Carstensen and Georg Dolzmann.
\newblock A posteriori error estimates for mixed {FEM} in elasticity.
\newblock {\em Numer. Math.}, 81(2):187--209, 1998.

\bibitem{CarstensenH_FOT}
Carsten Carstensen and Norbert Heuer.
\newblock A fractional-order trace-dev-div inequality.
\newblock {arXiv}:2403.01291, 2024.

\bibitem{CessenatD_98_AUW}
Olivier Cessenat and Bruno Despr{\'e}s.
\newblock Application of an ultra weak variational formulation of elliptic
  {PDE}s to the two-dimensional {H}elmholtz problem.
\newblock {\em SIAM J. Numer. Anal.}, 35(1):255--299, 1998.

\bibitem{DemkowiczG_11_CDP}
Leszek~F. Demkowicz and Jay Gopalakrishnan.
\newblock A class of discontinuous {Petrov-Galerkin} methods. {Part II}:
  {O}ptimal test functions.
\newblock {\em Numer. Methods Partial Differential Eq.}, 27:70--105, 2011.

\bibitem{DouglasR_85_GEM}
Jim Douglas, Jr. and Jean~E. Roberts.
\newblock Global estimates for mixed methods for second order elliptic
  equations.
\newblock {\em Math. Comp.}, 44(169):39--52, 1985.

\bibitem{Fuehrer_18_SDM}
Thomas F{\"u}hrer.
\newblock Superconvergence in a {DPG} method for an ultra-weak formulation.
\newblock {\em Comput. Math. Appl.}, 75(5):1705--1718, 2018.

\bibitem{FuehrerH_19_FDD}
Thomas F{\"u}hrer and Norbert Heuer.
\newblock Fully discrete {DPG} methods for the {Kirchhoff}--{Love} plate
  bending model.
\newblock {\em Comput. Methods Appl. Mech. Engrg.}, 343:550--571, 2019.

\bibitem{FuehrerHN_19_UFK}
Thomas F{\"u}hrer, Norbert Heuer, and Antti~H. Niemi.
\newblock An ultraweak formulation of the {Kirchhoff}--{Love} plate bending
  model and {DPG} approximation.
\newblock {\em Math. Comp.}, 88(318):1587--1619, 2019.

\bibitem{FuehrerHN_22_DMS}
Thomas F{\"u}hrer, Norbert Heuer, and Antti~H. Niemi.
\newblock A {DPG} method for shallow shells.
\newblock {\em Numer. Math.}, 152(1):76--99, 2022.

\bibitem{Grisvard_85_EPN}
Pierre Grisvard.
\newblock {\em Elliptic Problems in Nonsmooth Domains}.
\newblock Pitman Publishing Inc., Boston, 1985.

\bibitem{Grisvard_89_SEE}
Pierre Grisvard.
\newblock Singularit\'{e}s en elasticit\'{e}.
\newblock {\em Arch. Rational Mech. Anal.}, 107(2):157--180, 1989.

\bibitem{Hellan_67_AEP}
K\r{a}re Hellan.
\newblock Analysis of elastic plates in flexure by a simplified finite element
  method.
\newblock {\em Acta Polytech. Scand. Civ. Eng. Build. Constr. Ser. 46}, 1,
  1967.

\bibitem{Herrmann_67_FEB}
Leonard~R. Herrmann.
\newblock Finite-element bending analysis for plates.
\newblock {\em J. Eng. Mech. Div.}, 93(5):13--26, 1967.

\bibitem{Hu_15_FEA}
Jun Hu.
\newblock Finite element approximations of symmetric tensors on simplicial
  grids in {$\Bbb R^n$}: the higher order case.
\newblock {\em J. Comput. Math.}, 33(3):283--296, 2015.

\bibitem{HuZ_16_FEA}
Jun Hu and Shangyou Zhang.
\newblock Finite element approximations of symmetric tensors on simplicial
  grids in {$\Bbb{R}^n$}: the lower order case.
\newblock {\em Math. Models Methods Appl. Sci.}, 26(9):1649--1669, 2016.

\bibitem{Johnson_73_CMF}
Claes Johnson.
\newblock On the convergence of a mixed finite-element method for plate bending
  problems.
\newblock {\em Numer. Math.}, 21:43--62, 1973.

\bibitem{McLean_00_SES}
William McLean.
\newblock {\em Strongly Elliptic Systems and Boundary Integral Equations}.
\newblock Cambridge University Press, 2000.

\bibitem{PechsteinS_11_TDN}
Astrid Pechstein and Joachim Sch\"oberl.
\newblock Tangential-displacement and normal-normal-stress continuous mixed
  finite elements for elasticity.
\newblock {\em Math. Models Methods Appl. Sci.}, 21(8):1761--1782, 2011.

\bibitem{PechsteinS_12_AMF}
Astrid Pechstein and Joachim Sch\"oberl.
\newblock Anisotropic mixed finite elements for elasticity.
\newblock {\em Internat. J. Numer. Methods Engrg.}, 90(2):196--217, 2012.

\bibitem{PechsteinS_18_ATM}
Astrid~S. Pechstein and Joachim Sch\"{o}berl.
\newblock An analysis of the {TDNNS} method using natural norms.
\newblock {\em Numer. Math.}, 139(1):93--120, 2018.

\bibitem{RafetsederZ_18_DRK}
Katharina Rafetseder and Walter Zulehner.
\newblock A decomposition result for {K}irchhoff plate bending problems and a
  new discretization approach.
\newblock {\em SIAM J. Numer. Anal.}, 56(3):1961--1986, 2018.

\bibitem{RognesKL_09_EAH}
Marie~E. Rognes, Robert~C. Kirby, and Anders Logg.
\newblock Efficient assembly of {$H({\rm div})$} and {$H({\rm curl})$}
  conforming finite elements.
\newblock {\em SIAM J. Sci. Comput.}, 31(6):4130--4151, 2009/10.

\bibitem{Roessle_00_CSR}
Andreas R\"{o}ssle.
\newblock Corner singularities and regularity of weak solutions for the
  two-dimensional {L}am\'{e} equations on domains with angular corners.
\newblock {\em J. Elasticity}, 60(1):57--75, 2000.

\bibitem{SaendigRS_89_RBV}
Anna-Margarete S\"{a}ndig, Uwe Richter, and Rainer S\"{a}ndig.
\newblock The regularity of boundary value problems for the {L}am\'{e}
  equations in a polygonal domain.
\newblock {\em Rostock. Math. Kolloq.}, 36:21--50, 1989.

\bibitem{ScottZ_90_FEI}
L.~Ridgway Scott and Shangyou Zhang.
\newblock Finite element interpolation of nonsmooth functions satisfying
  boundary conditions.
\newblock {\em Math. Comp.}, 54(190):483--493, 1990.

\bibitem{Tartar_07_ISS}
Luc Tartar.
\newblock {\em An introduction to {S}obolev spaces and interpolation spaces},
  volume~3 of {\em Lecture Notes of the Unione Matematica Italiana}.
\newblock Springer, Berlin; UMI, Bologna, 2007.

\end{thebibliography}

\end{document}